\documentclass[12pt]{amsart}
\usepackage{amssymb}
\usepackage[all]{xy}

\newtheorem{theorem}{Theorem}[section]
\newtheorem{definition}[theorem]{Definition}
\newtheorem{proposition}[theorem]{Proposition}

\begin{document}

\title[Quantum groups]{Quantum groups, from a functional analysis perspective}

\author[T. Banica]{Teodor Banica}
\address{T.B.: Department of Mathematics, University of Cergy-Pontoise, F-95000 Cergy-Pontoise, France. {\tt teo.banica@gmail.com}}


\begin{abstract}
It is well-known that any compact Lie group appears as closed subgroup of a unitary group, $G\subset U_N$. The unitary group $U_N$ has a free analogue $U_N^+$, and the study of the closed quantum subgroups $G\subset U_N^+$ is a problem of general interest. We review here the basic tools for dealing with such quantum groups, with all the needed preliminaries included, and we discuss as well a number of more advanced topics.
\end{abstract}

\maketitle

{\em Dedicated to the memory of Stefan Banach.}

\bigskip

\tableofcontents

\section*{Introduction}

The unitary group $U_N$ has a free analogue $U_N^+$, whose standard coordinate functions $u_{ij}\in C(U_N^+)$ still form a unitary matrix, whose transpose is unitary too, but no longer commute. To be more precise, consider the following universal $C^*$-algebra:
$$C(U_N^+)=C^*\left((u_{ij})_{i,j=1,\ldots,N}\Big|u^*=u^{-1},u^t=\bar{u}^{-1}\right)$$

This algebra has a comultiplication, a counit and an antipode, constructed by using the universal property of $C(U_N^+)$, according to the following formulae:
$$\Delta(u_{ij})=\sum_ku_{ik}\otimes u_{kj}\quad,\quad\varepsilon(u_{ij})=\delta_{ij}\quad,\quad S(u_{ij})=u_{ji}^*$$

Thus $U_N^+$ is a compact quantum group in the sense of Woronowicz \cite{wo2}, \cite{wo3}. We are interested here in the closed subgroups $G\subset U_N^+$, the main examples being:

\medskip

\begin{enumerate}
\item The compact Lie groups, $G\subset U_N$.

\item The duals $G=\widehat{\Gamma}$ of the finitely generated groups $\Gamma=<g_1,\ldots,g_N>$. 

\item Deformations of the compact Lie groups, with parameter $q=-1$.

\item Liberations, half-liberations, quantum permutation groups, and more.
\end{enumerate}

\medskip

We will present the basic tools for dealing with such objects. The whole text is basically self-contained, notably containing an explanation of Woronowicz's papers \cite{wo2}, \cite{wo3}, with an updated terminology and notations, and under the assumption $G\subset U_N^+$, which simplifies a number of things. There is as well some supplementary material, concerning some additional methods, which are more recent and specialized. 

\bigskip

\noindent {\bf Disclaimer.} The closed subgroups $G\subset U_N^+$ that we discuss here do not cover several interesting examples of ``functional analytic'' quantum groups, such as:
\begin{enumerate}
\item The $q$-deformations of the compact Lie groups $G\subset U_N$, with general Drinfeld parameter $q\in\mathbb T$, or with general Woronowicz parameter $q\in\mathbb R$. For a functional analytic treatment of these quantum groups, we refer to \cite{wen}, \cite{wo2}.

\item The locally compact groups $G\subset GL_N(\mathbb C)$, and their various known quantum group versions. For some general theory and some examples here, struggling however with the existence of the Haar measure, we refer to \cite{kva}, \cite{pwo}.

\item The quantum group type objects underlying the combinatorics of certain ``specially patterned'' random matrices. The subject here is very wide, and for a number of such constructions, we refer for instance to \cite{fsn}, \cite{mpo}.
\end{enumerate}

\medskip

We believe, however, that there is a way of linking the present material with all these constructions. Regarding (1), a natural idea here, which has not been tried yet, would be that of studying first the $q$-deformations of the easy quantum groups, at $q=\pm i$. Regarding (2), the half-liberation machinery from \cite{bdu} applies a priori to the locally compact case as well, and this remains to be explored. Finally, regarding (3), some advances here should normally come along the lines of \cite{bne}, and of subsequent papers.

\bigskip

Regarding the possible applications of all this, the problem is open. Our belief is that the closed subgroups $G\subset U_N^+$, and the theory that has been developed for them, can be of help, in connection with a number of questions in quantum physics. Unfortunately, we have nothing concrete for the moment. This is a matter of time.

\bigskip

The paper is organized in 4 parts, as follows:

\medskip

\begin{enumerate}
\item Sections 1-2 are an introduction to the closed subgroups $G\subset U_N^+$, with the main examples ($O_N,O_N^*,O_N^+,U_N,U_N^*,U_N^+$) explained in detail.

\item Sections 3-4 are a presentation of Woronowicz's theory in \cite{wo1}, \cite{wo2}, with the main examples and their twists ($\bar{O}_N,\bar{O}_N^*,O_N^+,\bar{U}_N,\bar{U}_N^*,U_N^+$) worked out.

\item Sections 5-6 contain more specialized results, regarding the corresponding reflection groups $(H_N,H_N^*,H_N^+,K_N,K_N^*,K_N^+)$, and other quantum groups.

\item Sections 7-8 contain results regarding the closed subgroups $H\subset G$, the maximal tori $\widehat{\Gamma}\subset G$, and the matrix models $\pi:C(G)\to M_K(C(X))$.
\end{enumerate}

\bigskip

\noindent {\bf Acknowledgements.} I would like to thank Julien Bichon, Alex Chirvasitu, Beno\^it Collins, Steve Curran, Uwe Franz, Ion Nechita, Adam Skalski, Roland Speicher and Roland Vergnioux, for substantial joint work on the subject, and for their heavy influence on the point of view developed here. Thanks to Poufinette, too.

\section{Operator algebras}

In order to introduce the compact quantum groups, we will use the space/algebra correspondence coming from operator algebra theory. Let us begin with:

\begin{proposition}
Given a Hilbert space $H$, the linear operators $T:H\to H$ which are bounded, in the sense that $||T||=\sup_{||x||\leq1}||Tx||$ is finite, form a complex algebra with unit, denoted $B(H)$. This algebra has the following properties:
\begin{enumerate}
\item $B(H)$ is complete with respect to $||.||$, and so we have a Banach algebra. 

\item $B(H)$ has an involution $T\to T^*$, given by $<Tx,y>=<x,T^*y>$.
\end{enumerate}
In addition, the norm and the involution are related by the formula $||TT^*||=||T||^2$.
\end{proposition}

\begin{proof}
The fact that we have indeed an algebra follows from:
$$||S+T||\leq||S||+||T||\quad,\quad||\lambda T||=|\lambda|\cdot||T||\quad,\quad||ST||\leq||S||\cdot||T||$$

Regarding now (1), if $\{T_n\}\subset B(H)$ is Cauchy then $\{T_nx\}$ is Cauchy for any $x\in H$, so we can define the limit $T=\lim_{n\to\infty}T_n$ by setting $Tx=\lim_{n\to\infty}T_nx$.

As for (2), here the existence of $T^*$ comes from the fact that $\phi(x)=<Tx,y>$ being a linear map $H\to\mathbb C$, we must have $\phi(x)=<x,T^*y>$, for a certain vector $T^*y\in H$. Moreover, since this vector is unique, $T^*$ is unique too, and we have as well:
$$(S+T)^*=S^*+T^*\quad,\quad (\lambda T)^*=\bar{\lambda}T^*\quad,\quad (ST)^*=T^*S^*\quad,\quad (T^*)^*=T$$

Observe also that we have indeed $T^*\in B(H)$, because:
$$||T||=\sup_{||x||=1}\sup_{||y||=1}<Tx,y>=\sup_{||y||=1}\sup_{||x||=1}<x,T^*y>=||T^*||$$

Regarding the last assertion, we have $||TT^*||\leq||T||\cdot||T^*||=||T||^2$. Also, we have:
$$||T||^2=\sup_{||x||=1}|<Tx,Tx>|=\sup_{||x||=1}|<x,T^*Tx>|\leq||T^*T||$$

By replacing $T\to T^*$ we obtain from this $||T||^2\leq||TT^*||$, and we are done.
\end{proof}

In what follows we will be interested in the algebras of operators, rather than in the operators themselves. The axioms here, coming from Proposition 1.1, are as follows:

\begin{definition}
A unital $C^*$-algebra is a complex algebra with unit $A$, having a norm $a\to||a||$ which makes it into a Banach algebra (the Cauchy sequences converge), and having as well an involution $a\to a^*$, which satisfies $||aa^*||=||a||^2$, for any $a\in A$.
\end{definition}

As a basic example, $B(H)$ is a $C^*$-algebra. More generally, any closed $*$-subalgebra $A\subset B(H)$ is a $C^*$-algebra. The celebrated Gelfand-Naimark-Segal (GNS) theorem states that any $C^*$-algebra appears in fact in this way. We will be back to this later on.

One very interesting feature of Definition 1.2, making the link with several branches of abstract mathematics, comes from the following basic observation:

\begin{proposition}
If $X$ is an abstract compact space,  the algebra $C(X)$ of continuous functions $f:X\to\mathbb C$ is a $C^*$-algebra, with norm $||f||=\sup_{x\in X}|f(x)|$, and involution $f^*(x)=\overline{f(x)}$. This algebra is commutative: $fg=gf$, for any $f,g\in C(X)$.
\end{proposition}

\begin{proof}
Almost everything here is trivial. Observe also that we have indeed:
$$||ff^*||=\sup_{x\in X}|f(x)\overline{f(x)}|=\sup_{x\in X}|f(x)|^2=||f||^2$$

Finally, we have $fg=gf$, since $f(x)g(x)=g(x)f(x)$ for any $x\in X$.
\end{proof}

In order to work out the precise space/algebra correspondence coming from Proposition 1.3, we will need some basic spectral theory. Let us begin with:

\begin{definition}
Given a complex algebra $A$, the spectrum of an element $a\in A$ is
$$\sigma(a)=\left\{\lambda\in\mathbb C\big|a-\lambda\not\in A^{-1}\right\}$$
where $A^{-1}\subset A$ is the set of invertible elements.
\end{definition}

As a basic example, the spectrum of a usual matrix $M\in M_N(\mathbb C)$ is the collection of its eigenvalues. Also, the spectrum of a continuous function $f\in C(X)$ is its image.

Given an element $a\in A$, its spectral radius $\rho (a)$ is by definition the radius of the smallest disk centered at $0$ containing $\sigma(a)$. We have the following result:

\begin{proposition}
Let $A$ be a $C^*$-algebra.
\begin{enumerate}
\item The spectrum of a norm one element is in the unit disk.

\item The spectrum of a unitary element $(a^*=a^{-1}$) is on the unit circle. 

\item The spectrum of a self-adjoint element ($a=a^*$) consists of real numbers. 

\item The spectral radius of a normal element ($aa^*=a^*a$) is equal to its norm.
\end{enumerate}
\end{proposition}

\begin{proof}
Here (1) is clear, by using the formula $1/(1-x)=1+x+x^2+\ldots$ 

Regarding now the middle assertions, we can use here the elementary fact that if $f$ is a rational function having poles outside $\sigma(a)$, then $\sigma(f(a))=f(\sigma(a))$. Indeed, by using the functions $z^{-1}$ and $(z+it)/(z-it)$, we obtain the results.

From (1) we obtain $\rho(a)\leq ||a||$. For the converse, if we fix $\rho>\rho(a)$, we have:
$$\int_{|z|=\rho}\frac{z^n}{z -a}\,dz =\sum_{k=0}^\infty\left(\int_{|z|=\rho}z^{n-k-1}dz\right) a^k=a^n$$

By applying the norm and taking $n$-th roots we get $\rho\geq\lim {||a^n||^{1/n}}$.

In the case $a=a^*$ we have $||{a^n}||=||{a}||^n$ for any exponent of the form $n=2^k$, and by taking $n$-th roots we get $\rho\geq ||{a}||$. This gives the missing inequality $\rho(a)\geq ||a||$.

In the general case $aa^*=a^*a$ we have $a^n(a^n)^*=(aa^*)^n$, and we get $\rho(a)^2=\rho(aa^*)$. Now since $aa^*$ is self-adjoint, we get $\rho(aa^*)=||{a}||^2$, and we are done.
\end{proof}

We are now in position of proving a key result:

\begin{theorem}[Gelfand]
Any commutative $C^*$-algebra is the form $C(X)$, with its ``spectrum'' $X=Spec(A)$ appearing as the space of characters $\chi :A\to\mathbb C$.
\end{theorem}

\begin{proof}
Given a commutative $C^*$-algebra $A$, we can define indeed $X$ to be the set of characters $\chi :A\to\mathbb C$, with the topology making continuous all the evaluation maps $ev_a:\chi\to\chi(a)$. Then $X$ is a compact space, and $a\to ev_a$ is a morphism of algebras $ev:A\to C(X)$. We first prove that $ev$ is involutive. We use the following formula:
$$a=\frac{a+a^*}{2}-i\cdot\frac{i(a-a^*)}{2}$$

Thus it is enough to prove the equality $ev_{a^*}=ev_a^*$ for self-adjoint elements $a$. But this is the same as proving that $a=a^*$ implies that $ev_a$ is a real function, which is in turn true, because $ev_a(\chi)=\chi(a)$ is an element of $\sigma(a)$, contained in $\mathbb R$.

Since $A$ is commutative, each element is normal, so $ev$ is isometric, $||ev_a|| =\rho(a)=||{a}||$.

It remains to prove that $ev$ is surjective. But this follows from the Stone-Weierstrass theorem, because $ev(A)$ is a closed subalgebra of $C(X)$, which separates the points.
\end{proof}

In view of Gelfand's theorem, we can now formulate:

\begin{definition}
Given an arbitrary $C^*$-algebra $A$, we write $A=C(X)$, and call $X$ a noncommutative compact space. Equivalently, the category of the noncommutative compact spaces is the category of the $C^*$-algebras, with the arrows reversed.
\end{definition}

When $A$ is commutative, the space $X$ considered above exists indeed, as a Gelfand spectrum, $X=Spec(A)$. In general, $X$ is something rather abstract, and our philosophy here will be that of studying of course $A$, but formulating our results in terms of $X$. For instance whenever we have a morphism $\Phi:A\to B$, we will write $A=C(X),B=C(Y)$, and rather speak of the corresponding morphism $\phi:Y\to X$. And so on.

Finally, let us review the other fundamental result regarding the $C^*$-algebras, namely the representation theorem of Gelfand, Naimark and Segal. We will need:

\begin{proposition}
For an element $a\in A$, the following are equivalent:
\begin{enumerate}
\item $a$ is positive, in the sense that $\sigma(a)\subset[0,\infty)$.

\item $a=b^2$, for some $b\in A$ satisfying $b=b^*$.

\item $a=cc^*$, for some $c\in A$.
\end{enumerate}
\end{proposition}

\begin{proof}
Regarding $(1)\implies(2)$, observe that $\sigma(a)\subset\mathbb R$ implies $a=a^*$. Thus the algebra $<a>$ is commutative, and by using the Gelfand theorem, we can set $b=\sqrt{a}$. 

The implication $(2)\implies(3)$ is trivial, because we can set $c=b$. Observe that $(2)\implies(1)$ is clear too, because we have $\sigma(a)=\sigma(b^2)=\sigma(b)^2\subset\mathbb R^2=[0,\infty)$.

For $(3)\implies(1)$, we proceed by contradition. By multiplying $c$ by a suitable element of $<cc^*>$, we are led to the existence of an element $d\neq0$ satisfying $-dd^*\geq0$. By writing $d=x+iy$ with $x=x^*,y=y^*$ we have $dd^*+d^*d=2(x^2+y^2)\geq0$, and so $d^*d\geq0$. But this contradicts the elementary fact that $\sigma(de),\sigma(ed)$ must coincide outside $\{0\}$.
\end{proof}

We will need as well the following definition:

\begin{definition}
Consider the linear continuous maps $\varphi:A\to\mathbb C$, called states of $A$.
\begin{enumerate}
\item $\varphi$ is called positive when $a\geq0\implies\varphi(a)\geq0$.

\item $\varphi$ is called faithful and positive when $a\geq0,a\neq0\implies\varphi(a)>0$.
\end{enumerate}
\end{definition}

In the commutative case, $A=C(X)$, the states are of the form $\varphi(f)=\int_Xf(x)d\mu(x)$, with $\mu$ being positive/strictly positive in order for $\varphi$ to be positive/faithful and positive. In analogy with the fact that any compact space $X$ has a probability measure $\mu$, one can prove that any $C^*$-algebra $A$ has a faithful positive state $\varphi:A\to\mathbb C$. See \cite{ped}.

With these ingredients in hand, we can now state:

\begin{theorem}[GNS theorem]
Let $A$ be a $C^*$-algebra.
\begin{enumerate}
\item $A$ appears as a closed $*$-subalgebra $A\subset B(H)$, for some Hilbert space $H$. 

\item When $A$ is separable (usually the case), $H$ can be chosen to be separable.

\item When $A$ is finite dimensional, $H$ can be chosen to be finite dimensional. 
\end{enumerate}
\end{theorem}

\begin{proof}
In the commutative case, where $A=C(X)$, this algebra can be represented on $H=L^2(X)$, via $T_f(g)=fg$, provided that we have a probability measure on $X$. 

In general now, we can pick a faithful positive state $\varphi:A\to\mathbb C$, then let $H=l^2(A)$ be the completion of $A$ with respect to the scalar product $<a,b>=\varphi(ab^*)$, and finally represent $A$ on this space via $T_a(b)=ab$. For details here, we refer to \cite{ped}.
\end{proof}

\section{Quantum groups}

The quantum groups are abstract objects, generalizing the usual groups. The  most basic examples are the group duals. Let us recall indeed that associated to any discrete group $\Gamma$ is its group algebra $C^*(\Gamma)$, obtained as enveloping $C^*$-algebra of the usual group $*$-algebra $\mathbb C[\Gamma]=span(\Gamma)$. This algebra has a canonical trace, given by $tr(g)=\delta_{g,1}$ on the generators. With these conventions, we have the following result:

\begin{proposition}
Let $G$ be a compact abelian group, and $\Gamma=\widehat{G}$ be its Pontrjagin dual, formed by the characters $\chi:G\to\mathbb T$. We have then a Fourier transform isomorphism
$$C(G)\simeq C^*(\Gamma)$$
which transforms the comultiplication, counit and antipode of $C(G)$, given by
$$\Delta\varphi(g,h)=\varphi(gh)\quad,\quad\varepsilon(\varphi)=\varphi(1)\quad,\quad S\varphi(g)=\varphi(g^{-1})$$
into the comultiplication, counit and antipode of $C^*(\Gamma)$, given on  generators by:
$$\Delta(g)=g\otimes g\quad,\quad\varepsilon(g)=1\quad,\quad S(g)=g^{-1}$$
Moreover, the Haar integration over $G$ corresponds to the canonical trace of $C^*(\Gamma)$.
\end{proposition}

\begin{proof}
The first assertion follows from the basic properties of the Pontrjagin duality, and from the Gelfand theorem. The proof of the second assertion, regarding $\Delta,\varepsilon,S$, is routine. As for the last assertion, regarding the standard traces, this is clear too.
\end{proof}

In the non-abelian case now, there is still of lot of similarity between the algebras of type $C(G)$, and those of type $C^*(\Gamma)$. As a basic result here, we have:

\begin{proposition}
The comultiplication, counit and antipode of both the algebras $C(G)$, with $G$ compact group, and $C^*(\Gamma)$, with $\Gamma$ discrete group, satisfy:
\begin{enumerate}
\item $(\Delta\otimes id)\Delta=(id\otimes\Delta)\Delta$.

\item $(id\otimes\varepsilon)\Delta=(\varepsilon\otimes id)\Delta=id$.

\item $m(id\otimes S)\Delta=m(S\otimes id)\Delta=\varepsilon(.)1$
\end{enumerate}
In addition, in both cases the square of the antipode is the identity, $S^2=id$.
\end{proposition}

\begin{proof}
For the algebras of type $C(G)$, all the above formulae are well-known, coming via the Gelfand space/algebra correspondence from the following identities:
$$(gh)k=g(hk)\quad,\quad g1=1g=g\quad,\quad gg^{-1}=g^{-1}g=1\quad,\quad(g^{-1})^{-1}=g$$

As for the algebras of type $C^*(\Gamma)$, the formulae in the statement are all trivial.
\end{proof}

The above results suggest looking for a joint axiomatization of the algebras of type $C(G)$ and $C^*(\Gamma)$. However, this is quite tricky, due to some analytic issues with (3). Instead of getting into this, we will assume that $G$ is a compact Lie group, and that $\Gamma$ is a finitely generated group. These assumptions are related, due to:

\begin{proposition}
Given a compact abelian group $G$, and a discrete abelian group $\Gamma$, related by Pontrjagin duality, $G=\widehat{\Gamma}$ and $\Gamma=\widehat{G}$, the following are equivalent:
\begin{enumerate}
\item $G$ is a compact Lie group, $G\subset U_N$.

\item $\Gamma$ is finitely generated, $\Gamma=<g_1,\ldots,g_N>$.
\end{enumerate}
\end{proposition}

\begin{proof}
Assuming that we have a representation $\pi:G\subset U_N$, we know from Peter-Weyl theory that $\pi$ decomposes as a sum of characters, $\pi=g_1\oplus\ldots\oplus g_N$, and we obtain in this way generators for $\Gamma=\widehat{G}$. Conversely, assuming $\Gamma=<g_1,\ldots,g_N>$, the direct sum $\pi=g_1\oplus\ldots\oplus g_N$ is a faithful representation $\pi:G\subset U_N$ of the dual $G=\widehat{\Gamma}$.
\end{proof}

With these observations in hand, we can now go ahead, and formulate:

\begin{definition}
A Woronowicz algebra is a $C^*$-algebra $A$, given with a unitary matrix $u\in M_N(A)$ whose coefficients generate $A$, such that:
\begin{enumerate}
\item $\Delta(u_{ij})=\sum_ku_{ik}\otimes u_{kj}$ defines a morphism of $C^*$-algebras $A\to A\otimes A$.

\item $\varepsilon(u_{ij})=\delta_{ij}$ defines a morphism of $C^*$-algebras $A\to\mathbb C$.

\item $S(u_{ij})=u_{ji}^*$ defines a morphism of $C^*$-algebras $A\to A^{opp}$.
\end{enumerate}
In this case, we write $A=C(G)=C^*(\Gamma)$ and call $G$ a compact matrix quantum group, and $\Gamma$ a finitely generated discrete quantum group. Also, we write $G=\widehat{\Gamma}$, $\Gamma=\widehat{G}$.
\end{definition}

As a basic example, we have the algebra $A=C(G)$, with $G\subset U_N$, with the matrix formed by the standard coordinates $u_{ij}(g)=g_{ij}$. Indeed, $\Delta,\varepsilon,S$ are given by the formulae in the statement, due to the following formulae, valid for the unitary matrices:
$$(UV)_{ij}=\sum_kU_{ik}V_{kj}\quad,\quad 1_{ij}=\delta_{ij}\quad,\quad (U^{-1})_{ij}=\bar{U}_{ji}$$

The other basic example is the algebra $A=C^*(\Gamma)$, with $\Gamma=<g_1,\ldots,g_N>$, together with the matrix $u=diag(g_1,\ldots,g_N)$. Indeed, the comultiplication, counit and antipode of this algebra are by definition given by the formulae in the statement.

As a first general result now, regarding the above objects, we have:

\begin{proposition}
For a Woronowicz algebra, the maps $\Delta,\varepsilon,S$ satisfy the usual conditions for a comultiplication, counit and antipode, as in Proposition 2.2 above, at least on the dense $*$-subagebra $\mathcal A\subset A$ generated by the coordinates $u_{ij}$. 
\end{proposition}

\begin{proof}
This is clear, because all the three formulae in Proposition 2.2, as well as the supplementary formula $S^2=id$, are trivially satisfied on the generators $u_{ij}$. By linearity and multiplicativity, all these formulae are therefore satisfied on $\mathcal A$.
\end{proof}

At the level of the new examples now, we first have:

\begin{proposition}
The following universal algebras are Woronowicz algebras,
\begin{eqnarray*}
C(O_N^+)&=&C^*\left((u_{ij})_{i,j=1,\ldots,N}\Big|u=\bar{u},u^t=u^{-1}\right)\\
C(U_N^+)&=&C^*\left((u_{ij})_{i,j=1,\ldots,N}\Big|u^*=u^{-1},u^t=\bar{u}^{-1}\right)
\end{eqnarray*}
and so the underlying noncommutative spaces $O_N^+,U_N^+$ are compact quantum groups.
\end{proposition}

\begin{proof}
This follows from the fact that if a matrix $u$ is orthogonal/biunitary, then so are the matrices $u^\Delta_{ij}=\sum_ku_{ik}\otimes u_{kj}$, $u^\varepsilon_{ij}=\delta_{ij}$, $u^S_{ij}=u_{ji}^*$. Thus, we can define maps $\Delta,\varepsilon,S$ as in Definition 2.4, by using the universal properties of $C(O_N^+)$, $C(U_N^+)$. See \cite{wa1}.
\end{proof}

The basic properties of $O_N^+,U_N^+$ can be summarized as follows:

\begin{proposition}
The quantum groups $O_N^+,U_N^+$ have the following properties:
\begin{enumerate}
\item The closed subgroups $G\subset U_N^+$ are exactly the $N\times N$ compact quantum groups. As for the closed subgroups $G\subset O_N^+$, these are those satisfying $u=\bar{u}$.

\item We have ``liberation'' embeddings $O_N\subset O_N^+$ and $U_N\subset U_N^+$, obtained by dividing the algebras $C(O_N^+),C(U_N^+)$ by their respective commutator ideals.

\item We have as well embeddings $\widehat{L}_N\subset O_N^+$ and $\widehat{F}_N\subset U_N^+$, where $L_N$ is the free product of $N$ copies of $\mathbb Z_2$, and where $F_N$ is the free group on $N$ generators.
\end{enumerate}
\end{proposition}

\begin{proof}
Here (1) is clear from definitions, with the remark that, in the context of Definition 2.4 above, the formula $S(u_{ij})=u_{ji}^*$ shows that $\bar{u}$ must be unitary too. The assertion (2) follows from the Gelfand theorem. Finally, (3) follows from definitions, with the remark that in a group algebra we have $\bar{g}=g^{-1}$, and so $g=\bar{g}$ if and only if $g^2=1$.
\end{proof}

In order to construct more examples, we can look for intermediate objects for the inclusion $U_N\subset U_N^+$. There are several possible choices here, and the ``simplest'' ones, from a point of view that will be explained in detail later on, are as follows:

\begin{proposition}
We have intermediate quantum groups as follows:
\begin{enumerate}
\item $O_N\subset O_N^*\subset O_N^+$, obtained from $O_N^+$ by imposing to the variables $u_{ij}$ the half-commutation relations $abc=cba$. 

\item $U_N\subset U_N^*\subset U_N^+$, obtained from $U_N^+$ by imposing to the variables $u_{ij},u_{ij}^*$ the half-commutation relations $abc=cba$. 
\end{enumerate}
\end{proposition}

\begin{proof}
This is elementary, by using the fact that if the entries of $u=(u_{ij})$ half-commute, then so do the entries of $u^\Delta_{ij}=\sum_ku_{ik}\otimes u_{kj}$, $u^\varepsilon_{ij}=\delta_{ij}$, $u^S_{ij}=u_{ji}^*$. See \cite{bve}, \cite{bdu}.
\end{proof}

In order to distinguish the various quantum groups that we have, we can use:

\begin{proposition}
Given a closed subgroup $G\subset U_N^+$, consider its ``diagonal torus'', which is the closed subgroup $T\subset G$ constructed as follows:
$$C(T)=C(G)\Big/\left<u_{ij}=0\Big|\forall i\neq j\right>$$
This torus is then a group dual, $T=\widehat{\Lambda}$, where $\Lambda=<g_1,\ldots,g_N>$ is the discrete group generated by the elements $g_i=u_{ii}$, which are unitaries inside $C(T)$.
\end{proposition}

\begin{proof}
Since $u$ is unitary, its diagonal entries $g_i=u_{ii}$ are unitaries inside $C(T)$. Moreover, from $\Delta(u_{ij})=\sum_ku_{ik}\otimes u_{kj}$ we obtain $\Delta(g_i)=g_i\otimes g_i$, and so these unitaries $g_i\in C(T)$ are group-like. We conclude that we have $C(T)=C^*(\Lambda)$, and so $T=\widehat{\Lambda}$.
\end{proof}

Now back to our basic examples of quantum groups, we can formulate:

\begin{theorem}
The basic examples of compact quantum groups are as follows,
$$\xymatrix@R=15mm@C=15mm{
U_N\ar[r]&U_N^*\ar[r]&U_N^+\\
O_N\ar[r]\ar[u]&O_N^*\ar[r]\ar[u]&O_N^+\ar[u]}$$
and these quantum groups are non-isomorphic, distinguished by their diagonal tori.
\end{theorem}

\begin{proof}
The fact that we have quantum groups as above is something that we already know. Regarding now the diagonal tori, these are as follows, with $\circ$ being by the half-classical product, subject to relations $abc=cba$ between the standard generators:
$$\xymatrix@R=15mm@C=15mm{
\widehat{\mathbb Z^N}\ar[r]&\widehat{\mathbb Z^{\circ N}}\ar[r]&\widehat{\mathbb Z^{*N}}\\
\widehat{\mathbb Z_2^N}\ar[r]\ar[u]&\widehat{\mathbb Z_2^{\circ N}}\ar[r]\ar[u]&\widehat{\mathbb Z_2^{*N}}\ar[u]}$$

Now since the discrete groups in this diagram are clearly non-isomorphic, this shows that the corresponding quantum groups are non-isomorphic as well, as claimed.
\end{proof}

\section{Representation theory}

In order to reach to some more advanced insight into the structure of the closed subgroups $G\subset U_N^+$, we can use representation theory. Let us begin with:

\begin{definition}
Let $(A,u)$ be a Woronowicz algebra, and consider its dense $*$-subalgebra $\mathcal A\subset A$ of ``smooth elements'', generated by the standard coordinates $u_{ij}$.
\begin{enumerate}
\item A corepresentation of $A$ is a unitary matrix $r\in M_n(\mathcal A)$ satisfying $\Delta(r_{ij})=\sum_kr_{ik}\otimes r_{kj}$, $\varepsilon(r_{ij})=\delta_{ij}$ and $S(r_{ij})=r_{ji}^*$.

\item The corepresentations are subject to making sums, $r+p=diag(r,p)$, tensor products, $(r\otimes p)_{ia,jb}=r_{ij}p_{ab}$, and taking conjugates, $(\bar{r})_{ij}=r_{ij}^*$.

\item Given $r\in M_n(\mathcal A),p\in M_m(\mathcal A)$ we set $Hom(r,p)=\{T\in M_{m\times n}(\mathbb C)|Tr=pT\}$, and we use the notations $Fix(r)=Hom(1,r)$, and $End(r)=Hom(r,r)$.

\item Two corepresentations $r\in M_n(\mathcal A),p\in M_m(\mathcal A)$ are called equivalent, and we write $r\sim p$, when $n=m$, and $Hom(r,p)$ contains an invertible element.
\end{enumerate}
\end{definition}

For $A=C(G)$ we obtain in this way the representations of $G$, as a consequence of the Gelfand space/algebra correspondence. For $A=C^*(\Gamma)$, observe that any group element $g\in\Gamma$ is a one-dimensional corepresentation. We will see later on that, up to equivalence, each corepresentation of $C^*(\Gamma)$ splits as a direct sum of group elements.

We will need as well the following standard fact:

\begin{proposition}
The characters of corepresentations, given by $\chi_r=\sum_ir_{ii}$, satisfy:
$$\chi_{r+p}=\chi_r+\chi_p\quad,\quad\chi_{r\otimes p}=\chi_r\chi_p\quad,\quad\chi_{\bar{r}}=\chi_r^*$$
In addition, given two equivalent corepresentations, $r\sim p$, we have $\chi_r=\chi_p$.
\end{proposition}

\begin{proof}
The three formulae in the statement are all clear from definitions. Regarding now the last assertion, assuming that we have $r=T^{-1}pT$, we obtain:
$$\chi_r=Tr(r)=Tr(T^{-1}pT)=Tr(p)=\chi_p$$

We conclude that $r\sim p$ implies $\chi_r=\chi_p$, as claimed.
\end{proof}

In order to work out the analogue of the Peter-Weyl theory, we need to integrate over $G$. Things here are quite tricky, and best is to start with a definition, as follows:

\begin{definition}
The Haar integration of a Woronowicz algebra $A=C(G)$ is the unique positive unital tracial state $\int_G:A\to\mathbb C$ subject to the invariance conditions
$$\left(\int_G\otimes id\right)\Delta=\left(id\otimes\int_G\right)\Delta=\int_G(.)1$$
provided that such a state exists indeed, and is unique.
\end{definition}

As a basic example, given a compact Lie group $G\subset U_N$, the algebra $A=C(G)$ has indeed a Haar integration, which is the integration with respect to the Haar measure of $G$. Moreover, this latter measure can be obtained by starting with any probability measure on $G$, and then performing a Ces\`aro limit with respect to the convolution.

In analogy with this fact, we have the following result:

\begin{theorem}
Any Woronowicz algebra has a Haar integration, which can be constructed by starting with any faithful positive unital state $\varphi\in A^*$, and taking the Ces\`aro limit
$$\int_G=\lim_{n\to\infty}\frac{1}{n}\sum_{k=1}^n\varphi^{*k}$$
where the convolution operation for states is given by $\phi*\psi=(\phi\otimes\psi)\Delta$.
\end{theorem}

\begin{proof}
As already mentioned, this is well-known in the commutative case, $A=C(G)$. 

In the group dual case, $A=C^*(\Gamma)$, the result follows from Proposition 2.1 when $\Gamma$ is abelian, the Haar functional being given by $\int_{\widehat{\Gamma}}g=\delta_{g,1}$, for any $g\in\Gamma$. When $\Gamma$ is no longer abelian, the result still holds, and this is something well-known, and standard.

In the general case now, everything is quite tricky. We refer here to Woronowicz's paper \cite{wo1}, and also to the paper of Maes and Van Daele \cite{mva}, with the remark that our assumption $G\subset U_N^+$, which implies $S^2=id$, simplifies quite a number of things.
\end{proof}

Now back to representations, the basic integration result that we will need is:

\begin{proposition}
For any corepresentation $r\in M_n(\mathcal A)$, the operator 
$$P=\left(id\otimes\int_G\right)r\in M_n(\mathbb C)$$
is the orthogonal projection onto the space $Fix(r)=\{x\in\mathbb C^n|r(x)=x\otimes1\}$.
\end{proposition}

\begin{proof}
The invariance conditions in Definition 3.3 applied to $\varphi=r_{ij}$ read:
$$\sum_kr_{ik}P_{kj}=\sum_kP_{ik}r_{kj}=P_{ij}$$

Thus we have $rP=Pr=P$, and this gives the result. See \cite{wo1}.
\end{proof}

With these results in hand, we can now develop the Peter-Weyl theory:

\begin{theorem}
The corepresentations of a Woronowicz algebra $(A,u)$, taken modulo equivalence, are subject to the following Peter-Weyl type results:
\begin{enumerate}
\item Each representation appears as a sum of irreducible corepresentations. Also, each irreducible corepresentation appears inside a tensor product of $u,\bar{u}$.

\item The characters of irreducible corepresentations have norm $1$, and are pairwise orthogonal with respect to the scalar product $<a,b>=\int_Gab^*$.

\item The dense subalgebra $\mathcal A\subset A$ decomposes as a direct sum $\mathcal A=\oplus_{r\in Irr(A)}M_{\dim(r)}(\mathbb C)$, with the summands being pairwise orthogonal with respect to $<,>$.
\end{enumerate}
\end{theorem}

\begin{proof}
This follows as in the classical case, with the various statements involving the Haar functional basically coming from Proposition 3.5 above. See Woronowicz \cite{wo2}.
\end{proof}

As a first application of these results, we have:

\begin{theorem}
Let $A_{full}$ be the enveloping $C^*$-algebra of $\mathcal A$, and let $A_{red}$ be the quotient of $A$ by the null ideal of the Haar integration. The following are then equivalent:
\begin{enumerate}
\item The Haar functional of $A_{full}$ is faithful.

\item The projection map $A_{full}\to A_{red}$ is an isomorphism.

\item The counit map $\varepsilon:A_{full}\to\mathbb C$ factorizes through $A_{red}$.

\item Kesten criterion: $N\in\sigma(Re(\chi_u))$, inside the algebra $A_{red}$.
\end{enumerate}
If this is the case, we say that the underlying discrete quantum group $\Gamma$ is amenable.
\end{theorem}

\begin{proof}
This is well-known in the group dual case, $A=C^*(\Gamma)$, with $\Gamma$ being a usual discrete group. In general, the result follows by adapting the group dual case proof, by replacing where needed group elements by irreducible corepresentations. See \cite{ntu}. 
\end{proof}

The above results suggest that the ``combinatorics'' of a discrete quantum group $\Gamma$, appearing via $A=C^*(\Gamma)$, should come from the fusion rules on $Irr(A)$. This is indeed the case, and a whole theory can be developed here. See \cite{bve}, \cite{dpr}, \cite{fre}.

Let us explain now Woronowicz's Tannakian duality result from \cite{wo3}, in its ``soft'' form, worked out in \cite{mal}. The precise definition that we will need is:

\begin{definition}
The Tannakian category associated to a Woronowicz algebra $(A,u)$ is the collection $C=(C_{kl})$ of vector spaces
$$C_{kl}=Hom(u^{\otimes k},u^{\otimes l})$$
where the tensor powers, taken with respect to colored integers $k,l=\circ\bullet\bullet\circ\ldots$ are given by $u^\emptyset=1,u^\circ=u$, $u^\bullet=\bar{u}$ and multiplicativity.
\end{definition}

Observe that $C$ is indeed a tensor category, and more precisely is a tensor subcategory of the tensor category formed by the spaces $M_{kl}=\mathcal L(H^{\otimes k},H^{\otimes l})$, where $H\simeq\mathbb C^N$ is the Hilbert space where $u\in M_N(A)$ coacts, and where the tensor powers $H^{\otimes k}$ with $k$ colored integer are defined by $H^\emptyset=\mathbb C,H^\circ=H,H^\bullet=\bar{H}\simeq H$ and multiplicativity.

The Tannakian duality result, in its ``soft'' form, is as follows:

\begin{theorem}
Given a Woronowicz algebra $(A,u)$ with $u\in M_N(A)$, with associated Tannakian category $C=(C_{kl})$, we have
$$A=C(U_N^+)\big/\left<T\in Hom(v^{\otimes k},v^{\otimes l})\Big|\forall k,l,\forall T\in C_{kl}\right>$$
where $v$ denotes the fundamental corepresentation of $C(U_N^+)$.
\end{theorem}

\begin{proof}
If we denote by $A'$ the universal algebra on the right, we have a morphism $A'\to A$, because the canonical morphism $C(U_N^+)\to A$ factorizes through the ideal defining $A'$, by definition of the Tannakian category $C=(C_{kl})$. Conversely now, the fact that we have an arrow $A\to A'$ follows from Woronowicz's Tannakian duality results in \cite{wo3}, but there is as well  a direct, Hopf algebra proof for this, worked out in \cite{mal}.
\end{proof}

Generally speaking, knowing the Tannakian category of a Woronowicz algebra $A$ solves most of the fundamental problems regarding $A$. As an  illustration here, let us discuss the computation of the Haar state of $A$. The formula is very simple, as follows:

\begin{theorem}
Assuming that $A=C(G)$ has Tannakian category $C=(C_{kl})$, the Haar integration over $G$ is given by the Weingarten type formula
$$\int_Gu_{i_1j_1}^{e_1}\ldots u_{i_kj_k}^{e_k}=\sum_{\pi,\sigma\in D_k}\delta_\pi(i)\delta_\sigma(j)W_k(\pi,\sigma)$$
for any colored integer $k=e_1\ldots e_k$ and any multi-indices $i,j$, where $D_k$ is a linear basis of $C_{kk}$, $\delta_\pi(i)=<\pi,e_{i_1}\otimes\ldots\otimes e_{i_k}>$, and $W_k=G_k^{-1}$, with $G_k(\pi,\sigma)=<\pi,\sigma>$.
\end{theorem}

\begin{proof}
We know from Proposition 3.5 above that the integrals in the statement form altogether the orthogonal projection $P^e$ onto the space $Fix(u^{\otimes k})=span(D_k)$. 

By a standard linear algebra computation, it follows that we have $P=WE$, where $E(x)=\sum_{\pi\in D_k}<x,\xi_\pi>\xi_\pi$, and where $W$ is the inverse on $span(T_\pi|\pi\in D_k)$ of the restriction of $E$. But this restriction is the linear map given by $G_k$, and so $W$ is the linear map given by $W_k$, and this gives the formula in the statement. See \cite{bco}.
\end{proof}

\section{Basic examples}

We will show now that the Tannakian categories of the main 6 quantum groups appear in the simplest possible way: from certain ``categories'' of set-theoretic partitions.

In order to explain this material, let us begin with:

\begin{definition}
Let $P(k,l)$ be the set of partitions between an upper colored integer $k$, and a lower colored integer $l$. A set $D=\bigsqcup_{k,l}D(k,l)$ with $D(k,l)\subset P(k,l)$ is called a category of partitions when it is stable under the following operations:
\begin{enumerate}
\item The horizontal concatenation operation, $(\pi,\sigma)\to[\pi\sigma]$.

\item The vertical concatenation $(\pi,\sigma)\to[^\sigma_\pi]$, when the middle symbols match.

\item The upside-down turning operation $*$, with switching of the colors, $\circ\leftrightarrow\bullet$.
\end{enumerate}
\end{definition} 

Here, in the definition of the second operation, we agree that the connected components that can appear in the middle, when concatenating, are erased afterwards.

The relation with the Tannakian categories comes from:

\begin{proposition}
Each $\pi\in P(k,l)$ produces a linear map $T_\pi:(\mathbb C^N)^{\otimes k}\to(\mathbb C^N)^{\otimes l}$, 
$$T_\pi(e_{i_1}\otimes\ldots\otimes e_{i_k})=\sum_{j_1\ldots j_l}\delta_\pi\begin{pmatrix}i_1&\ldots&i_k\\ j_1&\ldots&j_l\end{pmatrix}e_{j_1}\otimes\ldots\otimes e_{j_l}$$
with the Kronecker type symbols $\delta_\pi\in\{0,1\}$ depending on whether the indices fit or not. The assignement $\pi\to T_\pi$ is categorical, in the sense that we have
$$T_\pi\otimes T_\sigma=T_{[\pi\sigma]}\quad,\quad T_\pi T_\sigma=N^{c(\pi,\sigma)}T_{[^\sigma_\pi]}\quad,\quad T_\pi^*=T_{\pi^*}$$
where $c(\pi,\sigma)$ are certain integers, coming from the erased components in the middle.
\end{proposition}

\begin{proof}
All three formulae are indeed elementary to establish. See \cite{bsp}.
\end{proof}

In relation now with the quantum groups, we have the following notion:

\begin{definition}
A closed subgroup $G\subset U_N^+$ is called easy when we have
$$Hom(u^{\otimes k},u^{\otimes l})=span\left(T_\pi\Big|\pi\in D(k,l)\right)$$
for any colored integers $k,l$, for a certain category of partitions $D\subset P$.
\end{definition}

As we will see in what follows, this formalism covers many interesting examples of groups and quantum groups. Let us first go back to the examples that we have. These quantum groups are all easy, coming from certain categories of pairings, as follows:

\begin{theorem}
The basic unitary quantum groups are all easy, as follows,
$$\xymatrix@R=15mm@C=15mm{
U_N\ar[r]&U_N^*\ar[r]&U_N^+\\
O_N\ar[r]\ar[u]&O_N^*\ar[r]\ar[u]&O_N^+\ar[u]}
\ \ \ \ \ \xymatrix@R=8mm@C=5mm{\\ :&\\&\\}\ \ 
\xymatrix@R=16mm@C=15mm{
\mathcal P_2\ar[d]&\mathcal P_2^*\ar[l]\ar[d]&\mathcal{NC}_2\ar[l]\ar[d]\\
P_2&P_2^*\ar[l]&NC_2\ar[l]}$$
with the corresponding categories of partitions being the following ones:
\begin{enumerate}
\item $P_2,NC_2$ are respectively the categories of pairings, and of noncrossing pairings, and $P_2^*$ is the category of pairings having the property that, when the legs are relabelled clockwise $\circ\bullet\circ\bullet\ldots$, each string connects $\circ-\bullet$.

\item $\mathcal P_2$ is the category of pairings which are ``matching'', in the sense that the vertical strings connect either $\circ-\circ$ or $\bullet-\bullet$, and the horizontal strings connect $\circ-\bullet$, and $\mathcal P_2^*=\mathcal P_2\cap P_2^*$ and $\mathcal{NC}_2=\mathcal P_2\cap NC_2$.
\end{enumerate}
\end{theorem}

\begin{proof}
The results for $O_N,U_N$ go back to Brauer's paper \cite{bra}, their free versions are worked out in \cite{bco}, and the half-liberated results are from \cite{bve}, \cite{bdu}, the idea being as follows:

(1) $U_N^+$ is defined via the relations $u^*=u^{-1},u^t=\bar{u}^{-1}$, which tell us that the operators $T_\pi$, with $\pi={\ }^{\,\cap}_{\circ\bullet}$ and $\pi={\ }^{\,\cap}_{\bullet\circ}$, must be in the associated Tannakian category $C$. We therefore obtain $C=span(T_\pi|\pi\in D)$, with $D=<{\ }^{\,\cap}_{\circ\bullet}\,\,,{\ }^{\,\cap}_{\bullet\circ}>={\mathcal NC}_2$, as claimed.

(2) $O_N^+\subset U_N^+$ is defined by imposing the relations $u_{ij}=\bar{u}_{ij}$, which tell us that the operators $T_\pi$, with $\pi=|^{\hskip-1.32mm\circ}_{\hskip-1.32mm\bullet}$ and $\pi=|_{\hskip-1.32mm\circ}^{\hskip-1.32mm\bullet}$, must be in the associated Tannakian category $C$. We therefore obtain $C=span(T_\pi|\pi\in D)$, with $D=<\mathcal{NC}_2,|^{\hskip-1.32mm\circ}_{\hskip-1.32mm\bullet},|_{\hskip-1.32mm\circ}^{\hskip-1.32mm\bullet}>=NC_2$, as claimed.

(3) $U_N\subset U_N^+$ is defined via the relations $[u_{ij},u_{kl}]=0$ and $[u_{ij},\bar{u}_{kl}]=0$, which tell us that the operators $T_\pi$, with $\pi={\slash\hskip-2.1mm\backslash}^{\hskip-2.5mm\circ\circ}_{\hskip-2.5mm\circ\circ}$ and $\pi={\slash\hskip-2.1mm\backslash}^{\hskip-2.5mm\circ\bullet}_{\hskip-2.5mm\bullet\circ}$, must be in the associated Tannakian category $C$. Thus $C=span(T_\pi|\pi\in D)$, with $D=<\mathcal{NC}_2,{\slash\hskip-2.1mm\backslash}^{\hskip-2.5mm\circ\circ}_{\hskip-2.5mm\circ\circ},{\slash\hskip-2.1mm\backslash}^{\hskip-2.5mm\circ\bullet}_{\hskip-2.5mm\bullet\circ}>=\mathcal P_2$, as claimed.

(4) Regarding now $U_N^*\subset U_N^+$, the corresponding Tannakian category is generated by the operators $T_\pi$, with $\pi={\slash\hskip-2.1mm\backslash\hskip-1.65mm|}$\,, taken with all the possible $2^3=8$ matching colorings. Since these latter 8 partitions generate the category $\mathcal P_2^*$, we obtain the result.

(5) In order to deal now with $O_N$, we can simply use the formula $O_N=O_N^+\cap U_N$. At the categorical level, this tells us that the associated Tannakian category is given by $C=span(T_\pi|\pi\in D)$, with $D=<NC_2,\mathcal P_2>=P_2$, as claimed.

(6) Finally, for $O_N^*$ we can proceed similarly, by using the formula $O_N^*=O_N^+\cap U_N^*$. At the categorical level, this tells us that the associated Tannakian category is given by $C=span(T_\pi|\pi\in D)$, with $D=<NC_2,\mathcal P_2^*>=P_2^*$, as claimed.
\end{proof}

Let us discuss now some more examples, which are of the same nature as the 6 basic ones. The idea is that we can twist the basic quantum groups, as follows:

\begin{proposition}
We have quantum groups as follows, obtained via the twisted commutation relations $ab=\pm ba$, and twisted half-commutation relations $abc=\pm cba$,
$$\xymatrix@R=15mm@C=15mm{
\bar{U}_N\ar[r]&\bar{U}_N^*\ar[r]&U_N^+\\
\bar{O}_N\ar[r]\ar[u]&\bar{O}_N^*\ar[r]\ar[u]&O_N^+\ar[u]}$$
where the signs for $\bar{U}_N$ correspond to anticommutation for distinct entries on rows and columns, and commutation otherwise, and the other signs come from functoriality.
\end{proposition}

\begin{proof}
This is clear indeed, by proceeding as in the proof of Proposition 2.6, with the signs for $\bar{U}_N^*$ being those producing an inclusion $\bar{U}_N\subset\bar{U}_N^*$, and with those for $\bar{O}_N,\bar{O}_N^*$ producing inclusions $\bar{O}_N\subset\bar{U}_N$ and $\bar{O}_N^*\subset\bar{U}_N^*$. For details here, we refer to \cite{ba1}.
\end{proof}

In order to study the easiness properties of these quantum groups, we will need:

\begin{proposition}
We have a signature map $\varepsilon:P_{even}\to\{-1,1\}$, given by $\varepsilon(\pi)=(-1)^c$, where $c$ is the number of switches needed to make $\pi$ noncrossing. In addition:
\begin{enumerate}
\item For $\pi\in Perm(k,k)\simeq S_k$, this is the usual signature.

\item For $\pi\in P_2$ we have $(-1)^c$, where $c$ is the number of crossings.

\item For $\pi\in P$ obtained from $\sigma\in NC_{even}$ by merging blocks, the signature is $1$.
\end{enumerate}
\end{proposition}

\begin{proof}
The fact that the number $c$ in the statement is well-defined modulo 2 is standard, and we refer here to \cite{ba1}. As for the remaining assertions, these are as well from \cite{ba1}:

(1) For $\pi\in Perm(k,k)$ the standard form is $\pi'=id$, and the passage $\pi\to id$ comes by composing with a number of transpositions, which gives the signature. 

(2) For a general $\pi\in P_2$, the standard form is of type $\pi'=|\ldots|^{\cup\ldots\cup}_{\cap\ldots\cap}$, and the passage $\pi\to\pi'$ requires $c$ mod 2 switches, where $c$ is the number of crossings. 

(3) For a partition $\pi\in P_{even}$ coming from $\sigma\in NC_{even}$ by merging a certain number $n$ of blocks, the fact that the signature is 1 follows by recurrence on $n$.
\end{proof}

We can make act the partitions in $P_{even}$ on tensors in a twisted way, as follows:

\begin{proposition}
Associated to any partition $\pi\in P_{even}(k,l)$ is the linear map
$$\bar{T}_\pi(e_{i_1}\otimes\ldots\otimes e_{i_k})=\sum_{\sigma\leq\pi}\varepsilon(\sigma)\sum_{j:\ker(^i_j)=\sigma}e_{j_1}\otimes\ldots\otimes e_{j_l}$$
and the assignement $\pi\to\bar{T}_\pi$ is categorical.
\end{proposition}

\begin{proof}
This is routine, by using ingredients from Proposition 4.6. See \cite{ba1}. 
\end{proof}

With these notions in hand, we can now investigate the twists, as follows:

\begin{theorem}
The quantum groups from Proposition 4.5 appear as Schur-Weyl twists of the quantum groups in Theorem 4.4, in the sense that for $\bar{G}$ we have
$$Hom(u^{\otimes k},u^{\otimes l})=span\left(\bar{T}_\pi\Big|\pi\in D(k,l)\right)$$
for any colored integers $k,l$, where $D\subset P_2\subset P_{even}$ is the category of partitions for $G$. In addition, the diagonal tori for $G,\bar{G}$ coincide.
\end{theorem}

\begin{proof}
All this is quite routine, by following the proof of Theorem 4.4 above, and adding signs where needed. The final assertion is clear as well. For details, see \cite{ba1}.
\end{proof}

As a first application of all this, we have:

\begin{theorem}
In the $N\to\infty$ limit, the law of the main character $\chi_u$ is as follows:
\begin{enumerate}
\item For $O_N,U_N$ we obtain Gaussian/complex Gaussian variables.

\item For $O_N^*,U_N^*$ we obtain ``squeezed'' versions of these variables.

\item For $O_N^+,U_N^+$ we obtain semicircular/circular variables.
\end{enumerate}
In addition, the asymptotic law of $\chi_u$ is invariant under Schur-Weyl twisting.
\end{theorem}

\begin{proof}
We know from the Peter-Weyl theory that the moments of $\chi_u$ are the dimensions of the spaces $Fix(u^{\otimes k})$. Now since with $N\to\infty$ the linear maps $T_\pi$ become linearly independent, the asymptotic moments of $\chi_u$ count the corresponding pairings, and modulo some standard facts from classical and free probability, this gives the result. See \cite{ba1}.
\end{proof}

At a more advanced level, the Weingarten integration formula from Theorem 3.10 takes a particularly simple form, the Gram matrix being given by $G_{kN}(\pi,\sigma)=N^{|\pi\vee\sigma|}$, where $|.|$ is the number of blocks. For applications of this formula, see \cite{ba1}, \cite{bco}, \cite{bcu}, \cite{bsp}. 

\section{Reflection groups}

The quantum groups that we considered so far, namely $O_N,U_N$ and their liberations and twists, are obviously of ``continuous'' nature. In order to have as well ``discrete'' examples, the idea will be that of looking at the corresponding quantum reflection groups.

Let us begin with a study of the quantum permutations. For this purpose, we will need the following functional analytic description of the usual symmetric group:

\begin{proposition}
Consider the symmetric group $S_N$.
\begin{enumerate}
\item The standard coordinates $v_{ij}\in C(S_N)$, coming from the embedding $S_N\subset O_N$ given by the permutation matrices, are given by $v_{ij}=\chi(\sigma|\sigma(j)=i)$.

\item The matrix $v=(v_{ij})$ is magic, in the sense that its entries are orthogonal projections, summing up to $1$ on each row and each column.

\item The algebra $C(S_N)$ is isomorphic to the universal commutative $C^*$-algebra generated by the entries of a $N\times N$ magic matrix.
\end{enumerate}
\end{proposition}

\begin{proof}
The assertions (1,2) are both clear. If we set $A=C^*_{comm}((w_{ij})_{i,j=1,\ldots,N}|w={\rm magic})$, we have a quotient map $A\to C(S_N)$, given by $w_{ij}\to v_{ij}$. On the other hand, by using the Gelfand theorem we can write $A=C(X)$, with $X$ being a compact space, and by using the coordinates $w_{ij}$ we have $X\subset O_N$, and then $X\subset S_N$. Thus we have as well a quotient map $C(S_N)\to A$ given by $v_{ij}\to w_{ij}$, and this gives (3). See Wang \cite{wa2}.
\end{proof}

With the above result in hand, we can now formulate:

\begin{proposition}
The following construction produces a Woronowicz algebra,
$$C(S_N^+)=C^*\left((u_{ij})_{i,j=1,\ldots,N}\Big|u={\rm magic}\right)$$
and the corresponding closed subgroup $S_N^+\subset O_N^+$ has the following properties:
\begin{enumerate}
\item $S_N^+$ is the universal compact quantum group acting on $\{1,\ldots,N\}$.

\item We have an embedding $S_N\subset S_N^+$, given by $u_{ij}\to\chi(\sigma|\sigma(j)=i)$.

\item This embedding is an isomorphism at $N=1,2,3$, but not at $N\geq4$.
\end{enumerate}
\end{proposition} 

\begin{proof}
Here the first assertion is standard, by using the elementary fact that if $u=(u_{ij})$ is magic, then so are the matrices $u^\Delta_{ij}=\sum_ku_{ik}\otimes u_{kj}$, $u^\varepsilon_{ij}=\delta_{ij}$, $u^S_{ij}=u_{ji}^*$. 

Regarding (1), given a closed subgroup $G\subset U_N^+$, it is straightforward to check that $\Phi(\delta_i)=\sum_ju_{ij}\otimes\delta_j$ defines a coaction map precisely when $u=(u_{ij})$ is a magic corepresentation of $C(G)$, and this gives the result. Also, (2) is clear from Proposition 5.1.

Regarding now (3), it is elementary to show that the entries of a $N\times N$ magic matrix, with $N\leq3$, must pairwise commute. At $N=4$ now, consider the following matrix:
$$U=\begin{pmatrix}
p&1-p&0&0\\
1-p&p&0&0\\
0&0&q&1-q\\
0&0&1-q&q
\end{pmatrix}$$ 

This matrix is magic for any two projections $p,q$, and if we choose these projections as for $<p,q>$ to be not commutative, and infinite dimensional, we conclude that $C(S_4^+)$ is not commutative and infinite dimensional as well, and so not isomorphic to $C(S_4)$.

Finally, at $N\geq5$ we can use the standard embedding $S_4^+\subset S_N^+$, obtained at the level of the corresponding magic matrices by $u\to diag(u,1_{N-4})$. See \cite{wa1}. 
\end{proof}

At the representation theory level, we have the following result:

\begin{theorem}
The quantum groups $S_N,S_N^+$ have the following properties:
\begin{enumerate}
\item They are both easy, the corresponding categories being the category of all partitions $P$, and the category of all noncrossing partitions $NC$. 

\item The corresponding asymptotic laws of the main characters are respectively the Poisson law, and the Marchenko-Pastur (or free Poisson) law.
\end{enumerate}
\end{theorem}

\begin{proof}
The assertions here are both standard, by performing an analysis similar to the one in the proof of Theorem 4.4, and of Theorem 4.9. For full details, see \cite{bbc}.
\end{proof}

Following now \cite{bb+}, we can formulate a key definition, as follows:

\begin{definition}
Given a closed subgroup $G\subset U_N^+$, we set
$$C(K)=C(G)\Big/\Big<u_{ij}u_{ij}^*=u_{ij}^*u_{ij}={\rm magic}\Big>$$ 
and we call $K\subset G$ the quantum reflection group associated to $G$.
\end{definition}

Here the fact that $K$ is indeed a closed subgroup of $G$ comes from the fact that if $u=(u_{ij})$ has the property that its entries are normal, and are such that $p_{ij}=u_{ij}u_{ij}^*$ form a magic matrix, then the same holds for $u^\Delta_{ij}=\sum_ku_{ik}\otimes u_{kj}$, $u^\varepsilon_{ij}=\delta_{ij}$, $u^S_{ij}=u_{ji}^*$. 

As basic examples here, for $G=O_N,U_N$ we obtain the groups $K=H_N,K_N$, which are the hyperoctahedral group $H_N=\mathbb Z_2\wr S_N$, and its complex version $K_N=\mathbb T\wr S_N$.

The free analogues of these results are as follows:

\begin{proposition}
The quantum reflection groups associated to $O_N^+,U_N^+$ are given by
$$C(K_N^+)=C(U_N^+)\Big/\Big<u_{ij}u_{ij}^*=u_{ij}^*u_{ij}={\rm magic}\Big>$$ 
and $H_N=K_N^+\cap O_N$, and we have decompositions $H_N^+=\mathbb Z_2\wr_*S_N^+$ and $K_N^+=\mathbb T\wr_*S_N^+$.
\end{proposition}

\begin{proof}
The fact that we have indeed decompositions as above is standard, and can be checked by constructing a pair of inverse isomorphisms, at the algebra level.
\end{proof}

More generally now, we have the following result:

\begin{theorem}
The basic $6$ basic unitary quantum groups and the $6$ basic twisted unitary groups have common quantum reflection groups, as follows,
$$\xymatrix@R=15.5mm@C=18mm{
K_N\ar[r]&K_N^*\ar[r]&K_N^+\\
H_N\ar[r]\ar[u]&H_N^*\ar[r]\ar[u]&H_N^+\ar[u]}$$
where $H_N^*,K_N^*$ are obtained from $H_N^+,K_N^+$ by intersecting with $U_N^*$. These quantum reflection groups are all easy, the corresponding categories of partitions being
$$\xymatrix@R=15.5mm@C=14mm{
\mathcal P_{even}\ar[d]&\mathcal P_{even}^*\ar[l]\ar[d]&\mathcal{NC}_{even}\ar[l]\ar[d]\\
P_{even}&P_{even}^*\ar[l]&NC_{even}\ar[l]}$$
where $P_{even},P_{even}^*,\mathcal P_{even}$ are the straightforward multi-block analogues of  the categories $P_2,P_2^*,\mathcal P_2$, and $\mathcal P_{even}^*=\mathcal P_{even}\cap P_{even}^*$, $NC_{even}=P_{even}\cap NC$, $\mathcal{NC}_{even}=\mathcal P_{even}\cap NC$.
\end{theorem}

\begin{proof}
Here the self-duality claim can be proved by using the various formulae in Proposition 4.6 above, and the easiness assertion is routine, by proceeding as in the proof of Theorem 4.4. For details regarding these results, we refer to \cite{ba1}, \cite{ba2}, \cite{bbc}.
\end{proof}

Our claim now is that, in the case of the above quantum groups, $G$ can be in fact reconstructed from $K$. In order to explain this material, we will need:

\begin{definition}
Given two closed subgroups $G,H\subset U_N^+$, with fundamental corepresentations denoted $u,v$, we construct a Tannakian category $C=(C_{kl})$ by setting
$$C_{kl}=Hom(u^{\otimes k},u^{\otimes l})\cap Hom(v^{\otimes k},v^{\otimes l})$$ 
and we let $<G,H>\subset U_N^+$ be the associated quantum group. That is, the corresponding fundamental corepresentation $w$ must satisfy $Hom(w^{\otimes k},w^{\otimes l})=C_{kl}$.
\end{definition}

Here the fact that $C=(C_{kl})$ is indeed a Tannakian category is clear from definitions. As for the notation $<,>$, this comes from the fact that in the classical case, where $G,H\subset U_N$, the closed subgroup of $U_N$ that we obtain is indeed the one generated by $G,H$.

In the easy case we have the following result:

\begin{proposition}
Assuming that $G,H\subset U_N^+$ are easy quantum groups, with corresponding categories of partitions $D,E\subset P$, we have:
\begin{enumerate}
\item $G\cap H$ is easy, with category of partitions $<D,E>$.

\item $<G,H>$ is easy, with category of partitions $D\cap E$.
\end{enumerate}
In addition, the same holds for the twisted easy quantum groups.
\end{proposition}

\begin{proof}
All the assertions are clear from the Tannakian correspondence between easy or twisted easy quantum groups, and categories of partitions.
\end{proof}

We can now go back to the quantum reflection groups, and formulate:

\begin{theorem}
Consider one of the basic easy quantum groups, $\dot{O}_N\subset G\subset U_N^+$, and let $H_N\subset K\subset K_N^+$ be its associated quantum reflection group. We have then 
$$G=<K,\dot{O}_N>$$
with the $<.>$ operation being the one from Definition 5.7 above.
\end{theorem}

\begin{proof}
This follows indeed by using the criterion in Proposition 5.8 above.
\end{proof}

There are many other interesting things that can be said about the reflection groups constructed above. We refer here to the survey paper \cite{bbc}, and to \cite{ba2}, \cite{bb+}, \cite{lta}.

\section{Classification results}

We discuss here various classification questions for the closed subgroups $G\subset U_N^+$, in the easy case, and in general. As a first, fundamental result, from \cite{bve}, we have:

\begin{theorem}
There is only one intermediate easy quantum group
$$O_N\subset G\subset O_N^+$$
namely the half-classical orthogonal group $O_N^*$.
\end{theorem}

\begin{proof}
We have to compute the categories of pairings $NC_2\subset D\subset P_2$.

\underline{Step I.} Let $\pi\in P_2-NC_2$, having $s\geq 4$ strings. Our claim is that:
\begin{enumerate}
\item If $\pi\in P_2-P_2^*$, there exists a semicircle capping $\pi'\in P_2-P_2^*$.

\item If $\pi\in P_2^*-NC_2$, there exists a semicircle capping $\pi'\in P_2^*-NC_2$.
\end{enumerate}

Indeed, both these assertions can be easily proved, by drawing pictures.

\underline{Step II.} Consider now a partition $\pi\in P_2(k,l)-NC_2(k,l)$. Our claim is that:
\begin{enumerate}
\item If $\pi\in P_2(k, l)-P_2^*(k,l)$ then $<\pi>=P_2$.

\item If $\pi\in P_2^*(k,l)-NC_2(k,l)$ then $<\pi>=P_2^*$.
\end{enumerate}

This can be indeed proved by recurrence on the number of strings, $s=(k+l)/2$, by using Step I, which provides us with a descent procedure $s\to s-1$, at any $s\geq4$.

\underline{Step III.} Finally, assume that we are given an easy quantum group $O_N\subset G\subset O_N^+$, coming from certain sets of pairings $D(k,l)\subset P_2(k,l)$. We have three cases:

(1) If $D\not\subset P_2^*$, we obtain $G=O_N$.

(2) If $D\subset P_2,D\not\subset NC_2$, we obtain $G=O_N^*$.

(3) If $D\subset NC_2$, we obtain $G=O_N^+$.
\end{proof}

Regarding now the arbitrary easy quantum groups $S_N\subset G\subset O_N^+$, we first have:

\begin{theorem}
The classical and free easy quantum groups are as follows,
$$
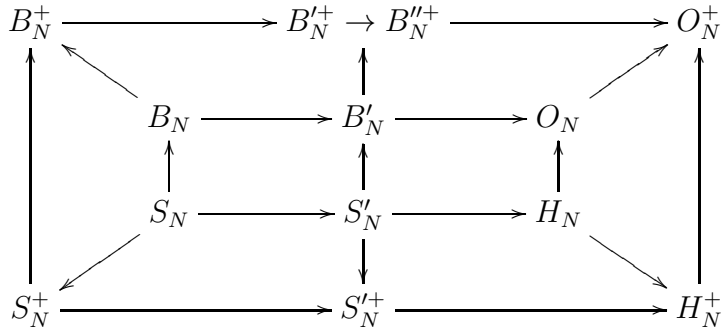
\xymatrix@R=6mm@C=10mm{
B_N^+\ar[rr]&&B_N'^+\to B_N''^+\ar[rr]&&O_N^+\\
&B_N\ar[r]\ar[ul]&B_N'\ar[u]\ar[r]&O_N\ar[ur]\\
&S_N\ar[r]\ar[u]\ar[dl]&S_N'\ar[u]\ar[d]\ar[r]&H_N\ar[u]\ar[dr]\\
S_N^+\ar[uuu]\ar[rr]&&S_N'^+\ar[rr]&&H_N^+\ar[uuu]}$$
where $S_N'=S_N\times\mathbb Z_2,B_N'=B_N\times\mathbb Z_2$, and $S_N'^+,B_N'^+,B_N''^+$ are their liberations.
\end{theorem}

\begin{proof}
The idea here is that of jointly classifying the ``classical'' categories of partitions $P_2\subset D\subset P$, and the ``free'' ones $NC_2\subset D\subset NC$. At the classical level this leads to 2 more groups, namely $S_N',B_N'$. See \cite{bsp}. At the free level we obtain 3 more quantum groups, $S_N'^+,B_N'^+,B_N''^+$, with the inclusion $B_N'^+\subset B_N''^+$ being best thought of as coming from an inclusion $B_N'\subset B_N''$, which happens to be an isomorphism. See \cite{bsp}.
\end{proof}

We can complete the above diagram with a number of intermediate liberations. The constructions and result here, which are quite technical, are as follows:

\begin{theorem}
The extra examples of easy quantum groups are as follows:
\begin{enumerate}
\item Half-liberations $O_N\subset O_N^*\subset O_N^*$ and $H_N\subset H_N^*\subset H_N^+$ and $B_N'\subset B_N''^*\subset B_N''^+$, obtained by imposing the half-commutation relations $abc=cba$.

\item A higher half-liberation $H_N^*\subset H_N^{[\infty]}\subset H_N^+$, obtained by imposing the relations $abc=0$, for any $a\neq c$ on the same row or column of $u$.

\item An uncountable family of intermediate quantum groups $S_N\subset H_N^\Gamma\subset H_N^{[\infty]}$, obtained from the quotients $\mathbb Z_2^{*\infty}\to\Gamma$ satisfying a certain uniformity condition.

\item A series of intermediate quantum groups $H_N^{[\infty]}\subset H_N^{\diamond k}\subset H_N^+$, obtained via the relations $[a_1\ldots a_{k-2}b^2a_{k-2}\ldots a_1,c^2]=0$.
\end{enumerate}
\end{theorem}

\begin{proof}
The construction and study of $O_N^*$ go back to \cite{bsp}, \cite{bve}, the quantum groups $H_N^*,H_N^{[\infty]}$ are from \cite{bcs}, and the constructions of the family $H_N^\Gamma$ and of the series $H_N^{\diamond k}$, as well as the proof of the classification result, are from \cite{rwe}.
\end{proof}

All this is quite technical, and in what follows, our purpose will be just of extending Theorem 6.1 above. In order to cut a bit from complexity, we will use:

\begin{proposition}
For a liberation operation of easy quantum groups $G_N\to G_N^+$, the following conditions are equivalent:
\begin{enumerate}
\item The category $P_2\subset D\subset P$ associated to $G=(G_N)$, or, equivalently, the category $NC_2\subset D\subset NC$ associated to $G^+=(G_N^+)$, is stable under removing blocks.

\item We have $G_N\cap U_K=G_K$, or, equivalently, $G_N^+\cap U_K^+=G_K^+$, for any $K\leq N$, where the embeddings $U_K\subset U_N$ and $U_K^+\subset U_N^+$ are the standard ones.

\item Each $G_N$ appears as lift of its projective version $G_N\to PG_N$, or, equivalently, each $G_N^+$ appears as lift of its projective version $G_N^+\to PG_N^+$.

\item The laws of truncated characters $\chi_t=\sum_{i=1}^{[tN]}u_{ii}$, with $t\in(0,1]$, for $G_N$ and $G_N^+$, form convolution/free convolution semigroups, in Bercovici-Pata bijection.
\end{enumerate}
If these conditions are satisfied, we call $G_N\to G_N^+$ a ``true'' liberation.
\end{proposition}

\begin{proof}
All this is well-known, basically going back to \cite{bsp}, the idea being that the implications $(1)\iff(2)\iff(3)$ are all elementary, and that $(1)\iff(4)$ follows by using the cumulant interpretation of the Bercovici-Pata bijection \cite{bep}, stating that ``the classical cumulants become via the bijection free cumulants''. See \cite{bss}, \cite{bsp}.
\end{proof}

We can now extend Theorem 6.1, as follows:

\begin{theorem}
There are precisely $4$ true liberations of orthogonal easy quantum groups, with the intermediate liberations, in the easy framework, being as follows:
\begin{enumerate}
\item $S_N\subset S_N^+$, with no intermediate object.

\item $O_N\subset O_N^+$, with $O_N^*$ as unique intermediate object.

\item $H_N\subset H_N^+$, with uncountably many intermediate objects.

\item $B_N\subset B_N^+$, with no intermediate object.
\end{enumerate}
\end{theorem}

\begin{proof}
The fact that the true liberations are indeed those in the statement follows from Proposition 6.4. As for the other statements, the proof here is routine. See \cite{bcs}.
\end{proof}

One interesting question is that of finding the intermediate quantum groups, not necessarily easy, for the above inclusions. There are several open problems here, the main one being that there is no intermediate quantum group $S_N\subset G\subset S_N^+$. See \cite{bbc}.

In the unitary case, the situation is considerably more complicated. We first have:

\begin{proposition}
Given $S\in\{1,2,3,\ldots,\infty\}$, we have an intermediate easy quantum group $O_N^+\subset O_{N,S}^+\subset O_{N,\infty}^+$, obtained by imposing the relations
$$a_{i_1}\ldots a_{i_S}=a_{i_1}^*\ldots a_{i_S}^*$$
to the standard coordinates of $O_{N,\infty}^+$, with the convention that at $S=\infty$ this relation dissapears. At $S=1$ we obtain in this way the quantum group $O_N^+$.
\end{proposition}

\begin{proof}
The relations in the statement are implemented by the following pairing:
$$\xymatrix@R=2mm@C=6mm{
\circ\ar@{-}[dd]&\ldots&\circ\ar@{-}[dd]\\
&(S)\\
\bullet&\ldots&\bullet}$$

But this gives the first assertion, and the last assertion is clear as well.
\end{proof}

Now by taking intersections, we are led to the following result:

\begin{theorem}
We have interesection/generation diagrams of easy quantum groups and of categories of pairings, in Tannakian correspondence, as follows,
$$\xymatrix@R=9mm@C=10mm{
U_N\ar[r]&U_N^*\ar[r]&U_N^+\\
O_{N,S}\ar[r]\ar[u]&O_{N,S}^*\ar[r]\ar[u]&O_{N,S}^+\ar[u]\\
O_N\ar[r]\ar[u]&O_N^*\ar[r]\ar[u]&O_N^+\ar[u]}
\ \ \ \ \ \xymatrix@R=12mm@C=5mm{\\ :&\\&\\}\ \ 
\xymatrix@R=9.5mm@C=10.5mm{
\mathcal P_2\ar[d]&\mathcal P_2^*\ar[l]\ar[d]&\mathcal{NC}_2\ar[l]\ar[d]\\
P_2^S\ar[d]&P_2^{S,*}\ar[l]\ar[d]&NC_2^S\ar[l]\ar[d]\\
P_2&P_2^*\ar[l]&NC_2\ar[l]}$$
where $P_2^S\subset P_2$ is the set of pairings which, when flattened, have the same number of $\circ,\bullet$ symbols, modulo $S$, and where the remaining categories appear as intersections.
\end{theorem}

\begin{proof}
We already know from Theorem 4.4 that the correspondence holds for the upper and lower rows. As for the middle row, the proof here is standard as well. 

Regarding the intersection/generation claim, which states that any square subdiagram $A\subset B,C\subset D$ is subject to the conditions $A=B\cap C, D=<B,C>$, this is clear for the pairings, and for the quantum groups this follows by using Proposition 5.8. See \cite{ba2}.
\end{proof}

The above quantum groups can be characterized as follows:

\begin{theorem}
The intermediate easy quantum groups $O_N\subset G\subset U_N^+$ satisfying 
$$G=<G_{class},G_{real}>$$
are precisely those constructed in Theorem 6.7 above.
\end{theorem}

\begin{proof}
According to \cite{twe}, the only easy quantum groups $O_N\subset G\subset U_N$ are the compact groups $O_{N,S}$, with $S\in\{2,3,\ldots,\infty\}$. Thus, we must have $G_{class}\in\{O_N,O_{N,S},U_N\}$. On the other hand, we know as well from Theorem 6.1 that we must have $G_{real}\in\{O_N,O_N^*,O_N^+\}$. Together with Theorem 6.7 above, this gives the result.
\end{proof}

The classification in the unitary case, and notably the classification of the intermediate easy quantum groups $O_N\subset G\subset U_N^+$, remains an open problem. See \cite{ba2}, \cite{bb2}, \cite{bdd}.

\section{Maximal tori}

In this section and in the next one we discuss various methods for the study of the closed subgroups $H\subset G$. The classical subgroups are easy to find, due to:

\begin{proposition}
Given a closed subgroup $G\subset U_N^+$, the classical subgroups $H\subset G$ are precisely the closed subgroups $H\subset G_{class}$, where $G_{class}=G\cap U_N$.
\end{proposition}

\begin{proof}
This is clear, because the formula $G_{class}=G\cap U_N$ means by definition that we have $C(G_{class})=C(G)/J$, where $J$ is the commutator ideal of $C(G)$.
\end{proof}

Let us investigate now the group dual subgroups $\widehat{\Lambda}\subset G$, whose knowledge is very useful as well. We have already met, in Proposition 2.9 above, the diagonal torus $T\subset G$. The construction there has the following generalization:

\begin{proposition}
Given a closed subgroup $G\subset U_N^+$ and a matrix $Q\in U_N$, we let $T_Q\subset G$ be the diagonal torus of $G$, with fundamental representation spinned by $Q$:
$$C(T_Q)=C(G)\Big/\left<(QuQ^*)_{ij}=0\Big|\forall i\neq j\right>$$
This torus is then a group dual, $T_Q=\widehat{\Lambda}_Q$, where $\Lambda_Q=<g_1,\ldots,g_N>$ is the discrete group generated by the elements $g_i=(QuQ^*)_{ii}$, which are unitaries inside $C(T_Q)$.
\end{proposition}

\begin{proof}
This follows indeed from Proposition 2.9 because, as said in the statement, $T_Q$ is by definition a diagonal torus, and so is a group dual, as indicated.
\end{proof}

With this notion in hand, we have the following result, coming from \cite{wo1}:

\begin{theorem}
Given a closed subgroup $G\subset U_N^+$, its group dual subgroups $\widehat{\Lambda}\subset G$ are exactly the quantum subgroups of type $\widehat{\Lambda}\subset T_Q$, with $Q\in U_N$.
\end{theorem}

\begin{proof}
This follows indeed from Woronowicz's results in \cite{wo1}.
\end{proof}

Summarizing, if we agree that the group duals are the correct generalization of the ``tori'' from the classical case, we are led to the following definition:

\begin{definition}
The maximal torus of a closed subgroup $G\subset U_N^+$ is the family
$$T=\left\{T_Q\subset G\big|Q\in U_N\right\}$$
of diagonal tori of $G$, parametrized by the various spinning matrices $Q\in U_N$.
\end{definition}

Our aim now is to show that $T$ plays indeed the role of a maximal torus for $G$. Let us first develop some general theory for these maximal tori. We first have:

\begin{proposition}
Given a closed subgroup $H\subset G$, the tori of $G,H$ are related by
$$T_Q(H)=T_Q(G)\cap H$$
with the intersection operation being the usual one, taken inside $G$.
\end{proposition}

\begin{proof}
Let $I=\ker(C(G)\to C(H))$. At $Q=1$ we have, indeed:
\begin{eqnarray*}
C(T_1(H))
&=&[C(G)/I]/<u_{ij}=0|i\neq j>\\
&=&[C(G)/<u_{ij}=0|i\neq j>]/I\\
&=&C(T_1(G)\cap H)
\end{eqnarray*}

In general, the proof is similar.
\end{proof}

Let us study the injectivity properties of the construction $G\to T$. We would like for instance to show that this construction is ``strictly increasing'' with respect to $\subset$. In other words, we would need a result stating that passing to a subgroup $H\subset G$ should decrease at least one of the tori $T_Q$. As a first statement in this direction, we have:

\begin{proposition}
Given a closed subgroup $G\subset U_N^+$, the following two constructions produce the same closed subgroup $G'\subset G$:
\begin{enumerate}
\item $G'=<T_Q|Q\in U_N>$, the closed subgroup generated by the tori $T_G\subset G$.

\item $G'=<\widehat{\Lambda}|\widehat{\Lambda}\subset G>$, the closed subgroup generated by the group duals $\widehat{\Lambda}\subset G$.
\end{enumerate}
\end{proposition} 

\begin{proof}
Let $G_1',G_2'\subset G$ be the two subgroups constructed above. Since any torus $T_Q$ is a group dual, we have $G_1'\subset G_2'$. Conversely, since any group dual $\widehat{\Lambda}\subset G$ appears as a subgroup of a certain torus, $\widehat{\Lambda}\subset T_Q\subset G$, we have $G_2'\subset G_1'$, and we are done.
\end{proof}

In view of this, it looks reasonable to formulate:

\begin{definition}
We say that $G\subset U_N^+$ is decomposable when the inclusion
$$G'=<T_Q|Q\in U_N>\subset G$$
constructed above is an equality, i.e. when $G$ is generated by its tori.
\end{definition}

At the level of basic examples of such quantum groups, we have:

\begin{theorem}
The following subgroups $G\subset U_N^+$ are decomposable:
\begin{enumerate}
\item The classical groups, $G\subset U_N\subset U_N^+$.

\item The group duals, $G=\widehat{\Gamma}\subset U_N^+$.
\end{enumerate}
\end{theorem}

\begin{proof}
This is elementary, but not trivial, the proofs being as follows:

(1) For $G=U_N$ we have $T_Q(G)=Q^*\mathbb T^NQ$, where $\mathbb T^N\subset U_N$ are the diagonal matrices, and so by Proposition 7.5 we obtain that for $G\subset U_N$ we have $T_Q(H)=Q^*\mathbb T^NQ\cap H$. Now since any group element $U\in H$ is diagonalizable, $U=Q^*DQ$ with $Q\in U_N,D\in\mathbb T^N$, we have $U\in T_Q(H)$ for this value of $Q\in U_N$, and this gives the result.

(2) This follows from Proposition 7.6 above, or directly from Theorem 7.3.
\end{proof}

In order to obtain more results, we can use Tannakian duality. We have:

\begin{proposition}
A closed subgroup $G\subset U_N^+$ is decomposable precisely when
$$\xi\in Fix(u_Q^{\otimes k}),\forall Q\in U_N\implies\xi\in Fix(u^{\otimes k})$$
where $u_Q=diag(g_1,\ldots,g_N)$ is the fundamental corepresentation of $T_Q\subset G$.
\end{proposition}

\begin{proof}
The Tannakian category associated to $G'=<T_Q|Q\in U_N>$ is given by:
$$Hom(v^{\otimes k},v^{\otimes l})=\bigcap_{Q\in U_N}Hom(u_Q^{\otimes k},u_Q^{\otimes l})$$

We conclude that the equality $G=G'$ is equivalent to the following collection of conditions, one for any pair $k,l$ of colored integers, as in \cite{mal}:
$$Hom(u^{\otimes k},u^{\otimes l})=\bigcap_{Q\in U_N}Hom(u_Q^{\otimes k},u_Q^{\otimes l})$$

Moreover, by Frobenius duality we can restrict the attention if we want to the spaces of fixed points, and this gives the conclusion in the statement. See \cite{mal}, \cite{wo2}. 
\end{proof}

In order to apply the above result, we can use the following formula, from \cite{bpa}:

\begin{proposition}
The intertwining formula $T\in Hom(u^{\otimes k},u^{\otimes l})$, with $u=QvQ^*$, where $v=diag(g_1,\ldots,g_N)$, is equivalent to the collection of conditions
$$(T^Q)_{j_1\ldots j_l,i_1\ldots i_k}\neq0\implies g_{i_1}\ldots g_{i_k}=g_{j_1}\ldots g_{j_l}$$
one for each choice of the multi-indices $i,j$, where $T^Q=(Q^*)^{\otimes l}TQ^{\otimes k}$. 
\end{proposition}

\begin{proof}
It is enough to prove the result at $Q=1$, and here we have:
\begin{eqnarray*}
T\in Hom(u^{\otimes k},u^{\otimes l})
&\iff&\sum_jT_{ji}e_j\otimes g_i=\sum_jT_{ji}e_j\otimes g_j,\forall i\\
&\iff&T_{ji}g_i=T_{ji}g_j,\forall i,j\\
&\iff&[T_{ji}\neq0\implies g_i=g_j],\forall i,j
\end{eqnarray*}

Thus we have obtained the relation in the statement, and we are done.
\end{proof}

Now by putting everything together, we obtain:

\begin{theorem}
A closed subgroup $G\subset U_N^+$ is decomposable precisely when
$$(T^Q)_{j_1\ldots j_l,i_1\ldots i_k}\neq0\implies g_{i_1}\ldots g_{i_k}=g_{j_1}\ldots g_{j_l}\ ({\rm inside}\ \Gamma_Q)$$
for any $Q\in U_N$ implies $T\in Hom(u^{\otimes k},u^{\otimes l})$, and this, for any $k,l$.
\end{theorem}

\begin{proof}
This follows indeed from Proposition 7.9 and Proposition 7.10 above.
\end{proof}

The above result is of course something quite technical, rather waiting to be applied in various concrete situations. In general, understanding the structure of the decomposable subgroups $G\subset U_N^+$, and the injectivity properties of the maximal torus construction $G\to T$, are definitely interesting problems, that we would like to raise here. 

There are as well several explicit conjectures regarding the maximal tori, the general idea being that the knowledge of $T$ solves most of the problems regarding $G$. See \cite{bpa}.

\section{Matrix models}

In this section we discuss matrix modelling questions for the closed subgroups $G\subset U_N^+$. We use the following matrix model formalism:

\begin{definition}
A matrix model for $C(G)$ is a morphism of $C^*$-algebras 
$$\pi:C(G)\to M_K(C(X))$$
with $X$ being a compact space, and with $K\in\mathbb N$. 
\end{definition}

The ``best'' kind of models are of course the faithful ones. However, since having an embedding $C(G)\subset M_K(C(X))$ forces the algebra $C(G)$ to be of type I, and so $G$ to be coamenable, we cannot expect such models to exist, in general. For instance $O_N^+$ with $N\geq3$, or $U_N^+$ with $N\geq2$, which are not coamenable, cannot have such models.

However, we have some interesting constructions of faithful models, as follows:

\begin{theorem}
The quantum groups $O_N^*,U_N^*$ are as follows:
\begin{enumerate}
\item We have an embedding $C(O_N^*)\subset M_2(C(U_N))$, mapping the coordinates $u_{ij}$ to antidiagonal self-adjoint matrices.

\item We have as well an embedding $C(U_N^*)\subset M_2(C(U_N\times U_N))$, obtained by using antidiagonal unitary matrices.
\end{enumerate}
\end{theorem}

\begin{proof}
Here the fact that we have a morphisms as in (1,2) is clear, because the antidiagonal matrices half-commute. The faithfulness of these models can be proved by using representation theory methods, and more specifically Theorem 4.4. See \cite{bb2}, \cite{bch}, \cite{bdu}.
\end{proof}

In order to deal now with the non-amenable case, let us go back to Definition 8.1. In the group dual case $G=\widehat{\Gamma}$, a matrix model $\pi:C^*(\Gamma)\to M_K(C(X))$ must come from a group representation $\rho:\Gamma\to C(X,U_K)$. Now observe that when $\rho$ is faithful, the induced representation $\pi$ is in general not faithful, its target algebra being finite dimensional. On the other hand, this representation ``reminds'' $\Gamma$. We say that $\pi$ is inner faithful.

We have in fact the following notions, coming from \cite{bb1}:

\begin{definition}
Let $\pi:C(G)\to M_K(C(X))$ be a matrix model. 
\begin{enumerate}
\item The Hopf image of $\pi$ is the smallest quotient Hopf $C^*$-algebra $C(G)\to C(H)$ producing a factorization of type $\pi:C(G)\to C(H)\to M_K(C(X))$.

\item When the inclusion $H\subset G$ is an isomorphism, i.e. when there is no non-trivial factorization as above, we say that $\pi$ is inner faithful.
\end{enumerate}
\end{definition}

In the case where $G=\widehat{\Gamma}$ is a group dual, $\pi$ must come from a group representation $\rho:\Gamma\to C(X,U_K)$, and the above factorization is simply the one obtained by taking the image, $\rho:\Gamma\to\Gamma'\subset C(X,U_K)$. Thus $\pi$ is inner faithful when $\Gamma\subset C(X,U_K)$.

Also, given a compact group $G$, and elements $g_1,\ldots,g_K\in G$, we have a representation $\pi:C(G)\to\mathbb C^K$, given by $f\to(f(g_1),\ldots,f(g_K))$. The minimal factorization of $\pi$ is then via $C(G')$, with $G'=\overline{<g_1,\ldots,g_K>}$, and $\pi$ is inner faithful when $G=G'$.

We refer to \cite{bb1}, \cite{bcv}, \cite{chi} for more on these facts, and for a number of related algebraic results. In what follows, we will rather use analytic techniques. Assume indeed that $X$ is a probability space. We have then the following result, from \cite{bfs}, \cite{wa3}:

\begin{theorem}
Given an inner faithful model $\pi:C(G)\to M_K(C(X))$, we have
$$\int_G=\lim_{k\to\infty}\frac{1}{k}\sum_{r=1}^k\int_G^r$$
where $\int_G^r=(\varphi\circ\pi)^{*r}$, with $\varphi=tr\otimes\int_X$ being the random matrix trace.
\end{theorem}

\begin{proof}
This was proved in \cite{bfs} in the case $X=\{.\}$, using idempotent state theory from \cite{fsk}. The general case was recently established in \cite{wa3}.
\end{proof}

The truncated integrals $\int_G^r$ can be evaluated as follows:

\begin{proposition}
Assuming that $\pi:C(G)\to M_K(C(X))$ maps $u_{ij}\to U_{ij}^x$, we have
$$\int_G^ru_{i_1j_1}^{e_1}\ldots u_{i_pj_p}^{e_p}=(T_e^r)_{i_1\ldots i_p,j_1\ldots j_p}$$
where $T_e\in M_{N^p}(\mathbb C)$ with $e\in\{1,*\}^p$ is given by $(T_e)_{i_1\ldots i_p,j_1\ldots j_p}=\int_Xtr(U_{i_1j_1}^{x,e_1}\ldots U_{i_pj_p}^{x,e_p})dx$.
\end{proposition}

\begin{proof}
This follows indeed from the definition of the various objects involved, namely from $\phi*\psi=(\phi\otimes\psi)\Delta$, and from $\Delta(u_{ij})=\sum_ku_{ik}\otimes u_{kj}$. See \cite{bb2}.
\end{proof}

As a first application, we can further investigate the faithful models, by using:

\begin{definition}
A stationary model for $C(G)$ is a random matrix model
$$\pi:C(G)\to M_K(C(X))$$
having the property $\int_G=(tr\otimes\int_X)\pi$.
\end{definition}

Observe that any stationary model is faithful. Indeed, the stationarity condition gives a factorization $\pi:C(G)\to C(G)_{red}\subset M_K(C(X))$, and since the image algebra $C(G)_{red}$ follows to be of type I, and therefore nuclear, $G$ must be co-amenable, and so $\pi$ must be faithful. For some background on these questions, we refer to \cite{ntu}.

As a useful criterion for the stationarity property, we have:

\begin{proposition}
For $\pi:C(G)\to M_K(C(X))$, the following are equivalent:
\begin{enumerate}
\item $Im(\pi)$ is a Hopf algebra, and $(tr\otimes\int_X)\pi$ is the Haar integration on it.

\item $\psi=(tr\otimes\int_X)\pi$ satisfies the idempotent state property $\psi*\psi=\psi$.

\item $T_e^2=T_e$, $\forall p\in\mathbb N$, $\forall e\in\{1,*\}^p$, where $(T_e)_{i_1\ldots i_p,j_1\ldots j_p}=(tr\otimes\int_X)(U_{i_1j_1}^{e_1}\ldots U_{i_pj_p}^{e_p})$.
\end{enumerate}
If these conditions are satisfied, we say that $\pi$ is stationary on its image.
\end{proposition}

\begin{proof}
Let us factorize our matrix model, as in Definition 8.3 above:
$$\pi:C(G)\to C(G')\to M_K(C(X))$$

Now observe that the conditions (1,2,3) only depend on the factorized representation $\pi':C(G')\to M_K(C(X))$. Thus, we can assume $G=G'$, which means that we can assume that $\pi$ is inner faithful. We can therefore use the formula in Theorem 8.4:
$$\int_G=\lim_{k\to\infty}\frac{1}{k}\sum_{r=1}^k\psi^{*r}$$

$(1)\implies(2)$ This is clear from definitions, because the Haar integration on any quantum group satisfies the equation $\psi*\psi=\psi$.

$(2)\implies(1)$ Assuming $\psi*\psi=\psi$, we have $\psi^{*r}=\psi$ for any $r\in\mathbb N$, and the above Ces\`aro limiting formula gives $\int_G=\psi$. By using now the amenability arguments explained after Definition 8.6, we obtain as well that $\pi$ is faithful, as desired.

In order to establish now $(2)\Longleftrightarrow(3)$, we use the formula in Proposition 8.5:
$$\psi^{*r}(u_{i_1j_1}^{e_1}\ldots u_{i_pj_p}^{e_p})=(T_e^r)_{i_1\ldots i_p,j_1\ldots j_p}$$

$(2)\implies(3)$ Assuming $\psi*\psi=\psi$, by using the above formula at $r=1,2$ we obtain that the matrices $T_e$ and $T_e^2$ have the same coefficients, and so they are equal.

$(3)\implies(2)$ Assuming $T_e^2=T_e$, by using the above formula at $r=1,2$ we obtain that the linear forms $\psi$ and $\psi*\psi$ coincide on any product of coefficients $u_{i_1j_1}^{e_1}\ldots u_{i_pj_p}^{e_p}$. Now since these coefficients span a dense subalgebra of $C(G)$, this gives the result.
\end{proof}

As a basic application of the above result, we have:

\begin{theorem}
The standard matrix models for the algebras $C(O_N^*),C(U_N^*)$, constructed in Theorem 8.2 above, are stationary.
\end{theorem}

\begin{proof}
This follows indeed from a routine Haar measure computation, which enhances the algebraic considerations from the proof of Theorem 8.2. See \cite{bch}.
\end{proof}

There are many other interesting examples of stationary models, including the Pauli matrix model for the algebra $C(S_4^+)$, discussed in \cite{bbc}. We refer to \cite{bch} and to subsequent papers for more on this subject, and for some recent results on the non-stationary case as well. There might be actually a relation here with lattice models too \cite{fad}.

Finally, many interesting questions arise in relation with Connes' noncommutative geometry \cite{con}, and we refer here to \cite{bdd}, \cite{cfk}, \cite{dgo}, \cite{gos}. Also, we refer to \cite{bcv}, \cite{chi}, \cite{vve}, \cite{vvo} for more specialized analytic aspects, and to \cite{ksp} and subsequent papers for free de Finetti theorems \cite{ksp}, in the spirit of Voiculescu's free probability theory \cite{vdn}.


\begin{thebibliography}{99}

\bibitem{ba1}T. Banica, Liberations and twists of real and complex spheres, {\em J. Geom. Phys.} {\bf 96} (2015), 1--25.

\bibitem{ba2}T. Banica, Unitary easy quantum groups: geometric aspects, {\em J. Geom. Phys.} {\bf 126} (2018), 127--147.

\bibitem{bb+}T. Banica, S.T. Belinschi, M. Capitaine and B. Collins, Free Bessel laws, {\em Canad. J. Math.} {\bf 63} (2011), 3--37.

\bibitem{bb1}T. Banica and J. Bichon, Hopf images and inner faithful representations, {\em Glasg. Math. J.} {\bf 52} (2010), 677--703.

\bibitem{bb2}T. Banica and J. Bichon, Matrix models for noncommutative algebraic manifolds, {\em J. Lond. Math. Soc.} {\bf 95} (2017), 519--540.

\bibitem{bbc}T. Banica, J. Bichon and B. Collins, Quantum permutation groups: a survey, {\em Banach Center Publ.} {\bf 78} (2007), 13--34.

\bibitem{bch}T. Banica and A. Chirvasitu, Thoma type results for discrete quantum groups, {\em Internat. J. Math.} {\bf 28} (2017), 1--23.

\bibitem{bco}T. Banica and B. Collins, Integration over compact quantum groups, {\em Publ. Res. Inst. Math. Sci.} {\bf 43} (2007), 277--302.

\bibitem{bcu}T. Banica and S. Curran, Decomposition results for Gram matrix determinants, {\em J. Math. Phys.} {\bf 51} (2010), 1--14.

\bibitem{bcs}T. Banica, S. Curran and R. Speicher, Classification results for easy quantum groups, {\em Pacific J. Math.} {\bf 247} (2010), 1--26.

\bibitem{bfs}T. Banica, U. Franz and A. Skalski, Idempotent states and the inner linearity property, {\em Bull. Pol. Acad. Sci. Math.} {\bf 60} (2012), 123--132.

\bibitem{bne}T. Banica and I. Nechita, Block-modified Wishart matrices and free Poisson laws, {\em Houston J. Math.} {\bf 41} (2015), 113--134.

\bibitem{bpa}T. Banica and I. Patri, Maximal torus theory for compact quantum groups, {\em Illinois J. Math.} {\bf 61} (2017), 151--170.

\bibitem{bss}T. Banica, A. Skalski and P.M. So\l tan, Noncommutative homogeneous spaces: the matrix case, {\em J. Geom. Phys.} {\bf 62} (2012), 1451--1466.

\bibitem{bsp}T. Banica and R. Speicher, Liberation of orthogonal Lie groups, {\em Adv. Math.} {\bf 222} (2009), 1461--1501.

\bibitem{bve}T. Banica and R. Vergnioux, Invariants of the half-liberated orthogonal group, {\em Ann. Inst. Fourier} {\bf 60} (2010), 2137--2164.

\bibitem{bep}H. Bercovici and V. Pata, Stable laws and domains of attraction in free probability theory, {\em Ann. of Math.} {\bf 149} (1999), 1023--1060.

\bibitem{bdd}J. Bhowmick, F. D'Andrea and L. Dabrowski, Quantum isometries of the finite noncommutative geometry of the standard model, {\em Comm. Math. Phys.} {\bf 307} (2011), 101--131.

\bibitem{bdu}J. Bichon and M. Dubois-Violette, Half-commutative orthogonal Hopf algebras, {\em Pacific J. Math.} {\bf 263} (2013), 13--28. 

\bibitem{bcv}M. Brannan, B. Collins and R. Vergnioux, The Connes embedding property for quantum group von Neumann algebras, {\em Trans. Amer. Math. Soc.} {\bf 369} (2017), 3799--3819. 

\bibitem{bra}R. Brauer, On algebras which are connected with the semisimple continuous groups, {\em Ann. of Math.} {\bf 38} (1937), 857--872.

\bibitem{chi}A. Chirvasitu, Residually finite quantum group algebras, {\em J. Funct. Anal.} {\bf 268} (2015), 3508--3533.

\bibitem{cfk}F. Cipriani, U. Franz and A. Kula, Symmetries of L\'evy processes on compact quantum groups, their Markov semigroups and potential theory, {\em J. Funct. Anal.} {\bf 266} (2014), 2789--2844. 

\bibitem{con}A. Connes, Noncommutative geometry, Academic Press (1994).

\bibitem{dpr}A. D'Andrea, C. Pinzari and S. Rossi, Polynomial growth for compact quantum groups, topological dimension and *-regularity of the Fourier algebra, preprint 2016.

\bibitem{dgo}B. Das and D. Goswami, Quantum Brownian motion on noncommutative manifolds: construction, deformation and exit times, {\em  Comm. Math. Phys.} {\bf 309} (2012), 193--228.

\bibitem{fad}L. Faddeev, Instructive history of the quantum inverse scattering method, {\em Acta Appl. Math.} {\bf 39} (1995), 69--84.

\bibitem{fsk}U. Franz and A. Skalski, On idempotent states on quantum groups, {\em J. Algebra} {\bf 322} (2009), 1774--1802.

\bibitem{fre}A. Freslon, On the partition approach to Schur-Weyl duality and free quantum groups, {\em Transform. Groups} {\bf 22} (2017), 707--751.

\bibitem{fsn}M. Fukuda and P. \'Sniady, Partial transpose of random quantum states: exact formulas and meanders, {\em J. Math. Phys.} {\bf 54} (2013), 1--31.

\bibitem{gos}D. Goswami, Quantum group of isometries in classical and  noncommutative geometry, {\em Comm. Math. Phys.} {\bf 285} (2009), 141--160.

\bibitem{ksp}C. K\"ostler, R. Speicher, A noncommutative de Finetti theorem: invariance under quantum permutations is equivalent to freeness with amalgamation, {\em Comm. Math. Phys.} {\bf 291} (2009), 473--490.

\bibitem{kva}J. Kustermans and S. Vaes, Locally compact quantum groups, {\em Ann. Sci. Ecole Norm. Sup.} {\bf 33} (2000), 837--934.

\bibitem{lta}F. Lemeux and P. Tarrago, Free wreath product quantum groups: the monoidal category, approximation properties and free probability, {\em J. Funct. Anal.} {\bf 270} (2016), 3828--3883.

\bibitem{mva}A. Maes and A. Van Daele, Notes on compact quantum groups, {\em Nieuw Arch. Wisk.} {\bf 16} (1998), 73--112. 

\bibitem{mal}S. Malacarne, Woronowicz's Tannaka-Krein duality and free orthogonal quantum groups, {\em Math. Scand.} {\bf 122} (2018), 151--160.

\bibitem{mpo}J.A. Mingo and M. Popa, Freeness and the partial transposes of Wishart random matrices, preprint 2017.

\bibitem{ntu}S. Neshveyev and L. Tuset, Compact quantum groups and their representation categories, SMF (2013).

\bibitem{ped}G.K. Pedersen, C$^*$-algebras and their automorphism groups, Academic Press (1979).

\bibitem{pwo}P. Podle\'s and S.L. Woronowicz, Quantum deformation of Lorentz group, {\em Comm. Math. Phys.} {\bf 130} (1990), 381--431.

\bibitem{rwe}S. Raum and M. Weber, The full classification of orthogonal easy quantum groups, {\em Comm. Math. Phys.} {\bf 341} (2016), 751--779.

\bibitem{twe}P. Tarrago and M. Weber, Unitary easy quantum groups: the free case and the group case, preprint 2015.

\bibitem{vve}S. Vaes and R. Vergnioux, The boundary of universal discrete quantum groups, exactness and factoriality, {\em Duke Math. J.} {\bf 140} (2007), 35--84.

\bibitem{vvo}R. Vergnioux and C. Voigt, The K-theory of free quantum groups, {\em Math. Ann.} {\bf 357} (2013), 355--400. 

\bibitem{vdn}D.V. Voiculescu, K.J. Dykema and A. Nica, Free random variables, AMS (1992).

\bibitem{wa1}S. Wang, Free products of compact quantum groups, {\em Comm. Math. Phys.} {\bf 167} (1995), 671--692.

\bibitem{wa2}S. Wang, Quantum symmetry groups of finite spaces, {\em Comm. Math. Phys.} {\bf 195} (1998), 195--211.

\bibitem{wa3}S. Wang, $L_p$-improving convolution operators on finite quantum groups, {\em Indiana Univ. Math. J.} {\bf 65} (2016), 1609--1637.

\bibitem{wen}H. Wenzl, C$^*$-tensor categories from quantum groups, {\em J. Amer. Math. Soc.} {\bf 11} (1998), 261--282.

\bibitem{wo1}S.L. Woronowicz, Twisted $SU(2)$ group. An example of a
non-commutative differential calculus, {\em Publ. Res. Inst. Math. Sci.} {\bf 23} (1987), 117--181.

\bibitem{wo2}S.L. Woronowicz, Compact matrix pseudogroups, {\em Comm. Math. Phys.} {\bf 111} (1987), 613--665.

\bibitem{wo3}S.L. Woronowicz, Tannaka-Krein duality for compact matrix pseudogroups. Twisted SU(N) groups, {\em Invent. Math.} {\bf 93} (1988), 35--76.

\end{thebibliography}
\end{document}